%% file: damiro1_secondo.tex
\documentclass[12pt,reqno,intlimits,a4paper,english]{amsart}

\usepackage[T1]{fontenc}
\usepackage{graphicx,pstricks}
\usepackage{calc}
\usepackage{babel}
\usepackage{amssymb,amsmath}

\numberwithin{equation}{section}
\newcounter{hours}\newcounter{minutes}
\newcommand{\hourandminute}{\setcounter{hours}{\time/60}%
\setcounter{minutes}{\time-\value{hours}*60}%
\thehours.\theminutes}
\newcommand{\Versione}{\jobname\ \today\ \hourandminute}
\newlength{\Indent}
\newlength{\Parskip}
\theoremstyle{plain}
\newtheorem{thm}{Theorem}[section]     
\newtheorem{prop}[thm]{Proposition}

\theoremstyle{remark}

\newtheorem{Remark}[thm]{Remark}
\newenvironment{remark}{\begin{Remark}}{\qed\end{Remark}}

\theoremstyle{definition}
\newtheorem{defin}[thm]{Definition}

\def\dato{\Psi}

\def\X{{\mathcal X}}

\newcommand{\Cper}{\mathfrak{C}}
\newcommand{\R}{\boldsymbol{R}}
\newcommand{\RN}{\R^{N}}

\newcommand{\di}{\,\text{\rmfamily\upshape d}}

\newcommand{\pder}[2]{\frac{\partial #1}
            {\partial #2}}

\newcommand{\normaa}[1]{|||{#1}|||}
\newcommand{\normab}[1]{||||{#1}||||}
\newcommand{\oo}{\infty}
\newcommand{\Om}{\varOmega}
\newcommand{\eps}{\varepsilon}
\newcommand{\Norma}[2]{{|\kern-1pt|\kern-1pt|#1|\kern-1pt|\kern-1pt|_{#2}}}

\newcommand{\defeq}{:=}

\newdimen\dimux
\def\xxintot{\int\limits_{0}^{t} \dimux=10pt \kern-10pt\raise4pt
          \hbox to7pt{\hrulefill} \advance\dimux by -7pt \kern+\dimux}
\def\xxintl#1#2{\int\limits_{#1} \dimux=#2pt \kern-#2pt\raise4pt
                \hbox to7pt{\hrulefill} \advance\dimux by -7pt \kern+\dimux}
\def\xxiintl#1#2{\iint\limits_{#1} \dimux=#2pt \kern-#2pt\raise4pt
                \hbox to7pt{\hrulefill} \advance\dimux by -7pt \kern+\dimux}
\DeclareMathOperator{\Div}{div}

\newcommand{\ZZ}{\boldsymbol{Z}}
\newcommand{\NN}{\boldsymbol{N}}
\newcommand{\CC}{\mathcal{C}}

\newcommand{\dfint}{\sigma_{1}}
\newcommand{\dfout}{\sigma_{2}}
\newcommand{\dfboth}{\sigma}
\newcommand{\dfbotheps}{\sigma^\eps}
\newcommand{\dfav}{\sigma_{0}}

\newcommand{\Per}{E}
\newcommand{\Omint}{\Om_{1}^{\eps}}
\newcommand{\Omout}{\Om_{2}^{\eps}}

\newcommand{\Memb}{\varGamma^{\eps}}

\newcommand{\Perint}{\Per_{1}}
\newcommand{\Perout}{\Per_{2}}
\newcommand{\Permemb}{\varGamma}

\newcommand{\vi}{\text{\rmfamily\upshape v}}

\newcommand{\spazioaaa}{\CC^0([0,T];H^1(\Om))\times \CC^0([0,T];L^2(\Om;{\X}^1_\#(Y)))}

\newcommand{\spaziob}{\CC^0([0,1];L^2(\Om\times Y))}

\newcommand{\spaziobbbdiesis}{\CC^0_\#([0,1];{\X}^1(\Om_\eps))}
\newcommand{\spazioc}{\CC^0([0,1];L^2(\Om\times\Permemb))}

\newcommand{\spazioee}{H^1(\Om)\times L^2(\Om;{\X}^1_\#(Y))}

\newcommand{\spaziogperaa}{\CC_\#^0([0,1];H^1(\Om))\times \CC_\#^0([0,1];L^2(\Om;{\X}^1_\#(Y))}

\newcommand{\spaziogperoaa}{\CC_\#^0([0,1];H^1_0(\Om))\times \CC_\#^0([0,1];L^2(\Om;{\X}^1_\#(Y))}
\newcommand{\spaziogzero}{\CC_c^0(\R;H^1_0(\Om))\times \CC_c^0(\R;L^2(\Om;{\X}^1_\#(Y))}

\begin{document}

\title{Asymptotic decay under nonlinear and noncoercive dissipative
effects for electrical conduction in biological tissues}

\author{M. Amar -- D. Andreucci -- R. Gianni\\
\hfill \\
Dipartimento di Scienze di Base e Applicate per l'Ingegneria\\
Universit\`a di Roma ``La Sapienza"\\
Via A. Scarpa 16, 00161 Roma, Italy
}

\begin{abstract}
We consider a nonlinear model for electrical conduction in biological tissues. The
nonlinearity appears in the interface condition prescribed on the cell membrane.

The purpose of this paper is proving asymptotic convergence for large times to a periodic
solution when time-periodic boundary data are assigned.
The novelty here is that we allow the nonlinearity to be noncoercive. We consider
both the homogenized and the non-homogenized version of the problem.
\bigskip

\textsc{Keywords:}
Asymptotic decay, stability, nonlinear homogenization, two-scale techniques, electrical impedance tomography.

\textsc{AMS-MSC:}
35B40, 35B27, 45K05, 92C55

\end{abstract}

\maketitle
\pagestyle{myheadings}
\thispagestyle{plain}
\markboth{M. AMAR, D. ANDREUCCI, AND R. GIANNI}{ASYMPTOTIC ...}

\section{Introduction}\label{s:introduction}

We study here a problem arising in electrical conduction in biological tissues
with the purpose of obtaining some useful results for applications in
electrical tomography, see \cite{Amar:Andreucci:Bisegna:Gianni:2003b}, \cite{Amar:Andreucci:Bisegna:Gianni:2003a},
\cite{Amar:Andreucci:Bisegna:Gianni:2004b},
\cite{Amar:Andreucci:Bisegna:Gianni:2004a},
\cite{Amar:Andreucci:Bisegna:Gianni:2005},
\cite{Amar:Andreucci:Bisegna:Gianni:2006a},
\cite{Amar:Andreucci:Bisegna:Gianni:2009},
\cite{Amar:Andreucci:Bisegna:Gianni:2009a},
\cite{Amar:Andreucci:Bisegna:Gianni:2010},
\cite{Amar:Andreucci:Bisegna:Gianni:2013},
\cite{Amar:Andreucci:Gianni:2014a}.
Our interest in this framework is motivated by the fact that
composite materials have widespread applications in science and technology
and, for this reason, they have been extensively studied especially using
homogenization techniques.

From a physical point of view our problem consists in the study of the
electrical currents crossing a living tissue when an electrical potential is
applied at the boundary (see \cite{Bisegna:Caruso:Lebon:2000}, \cite{Bronzino:1999},
\cite{DeLorenzo:1997}, \cite{Foster:Schwan:1989}, \cite{Kandell:2000}).
Here the living tissue is regarded as a composite
periodic domain made of extracellular and intracellular materials (both
assumed to be conductive, possibly with different conductivities) separated
by a lipidic membrane which experiments prove to exhibit both conductive
(due to ionic channels in the membrane) and capacitive behavior.
The periodic microstructure calls for the use of an homogenization
technique. Among the wide literature on this topic, we recall for instance
\cite{Allaire:1992}, \cite{Allaire:Briane:1996}, \cite{Allaire:Damlamian:Hornung:1995}, \cite{Auriault:Ene:1994},
\cite{Bellieud:Bouchitte:1998}, \cite{Bellieud:Bouchitte:2002},
\cite{Cioranescu:Donato:1988}, \cite{Clark:Packer:1997},
\cite{Hummel:2000}, \cite{Krassowska:Neu:1993}, \cite{Lene:Leguillon:1981}, \cite{Lipton:1998}, \cite{Malte:Bohm:2008},
\cite{Timofte:2013}.
As a result of the homogenization procedure we obtain a system of partial differential
equations satisfied by the macroscopic electrical potential $u$, which is the
limit of the electrical potential $u_{\eps }$ in the tissue as $\eps $
(the characteristic length of the cell) tends to zero.

Different scalings may appear in this homogenization procedure and they are studied in
\cite{Amar:Andreucci:Bisegna:Gianni:2006a} and \cite{Amar:Andreucci:Bisegna:Gianni:2013}.
We study here further developments of the model proposed in
\cite{Amar:Andreucci:Bisegna:Gianni:2003b},
\cite{Amar:Andreucci:Bisegna:Gianni:2003a},
\cite{Amar:Andreucci:Bisegna:Gianni:2004a},
\cite{Amar:Andreucci:Bisegna:Gianni:2006a},
\cite{Amar:Andreucci:Bisegna:Gianni:2010},
\cite{Amar:Andreucci:Bisegna:Gianni:2009a},
\cite{Amar:Andreucci:Bisegna:Gianni:2013},
where the magnetic field is neglected (as
suggested by experimental evidence) and the potential $u_{\eps }$ is assumed to
satisfy an elliptic equation both in the intracellular and in the
extracellular domain (see, \eqref{eq:PDEin} below) while, on the membranes it satisfies the
equation
\begin{equation*}
\frac{\alpha}{\eps}\frac{\partial}{\partial t}[u_{\eps}] +
   f\left(\frac{[u_{\eps}]}{\eps}\right) =\dfbotheps \nabla u_{\eps} \cdot \nu_{\eps}
\end{equation*}
where $\left[ u_{\eps}\right] $ denotes the jump of the potential
across the membranes and $\dfbotheps \nabla u_{\eps} \cdot \nu_{\eps}$
is the current crossing the membranes.
From a mathematical point of view a big difference does exist between the case of
linear $f$ and the nonlinear case, as already pointed out in \cite{Amar:Andreucci:Bisegna:Gianni:2013}
and \cite{Amar:Andreucci:Gianni:2014a}.

At least in the linear case, the asymptotic behavior of the potentials $u_\eps$ and $u$ is
crucial in order to validate the phenomenological model employed in bioimpedance tomography
devices, which currently relies on the use of complex elliptic equations, see
\cite{Amar:Andreucci:Bisegna:Gianni:2009}--\cite{Amar:Andreucci:Bisegna:Gianni:2010}.

Motivated by the previous considerations, in \cite{Amar:Andreucci:Gianni:2014a} and in this paper we
investigate the behavior as $t\to + \infty$ of the nonlinear problem introduced
in \cite{Amar:Andreucci:Bisegna:Gianni:2013}.

In \cite{Amar:Andreucci:Gianni:2014a}, we proved that, if periodic boundary data are assigned and $f$ is
coercive in the following sense
\begin{equation}\label{eq:intro1}
f\in\CC^1(\Om)\,,\qquad\quad f^\prime(s)\geq \kappa >0\,,\qquad \forall s\in\R\,,
\end{equation}
for a suitable $\kappa>0$, then the solution of the $\eps $-problem converges as $t\to
+\infty $ to a periodic function solving a suitable system of equations. In
that case such a convergence was proved to be exponential.
A similar asymptotic exponential behavior was proved for the solution of the
homogenized problem. Similar results in different frameworks can be found in \cite{Fabrizio:Lazzari:1986},
\cite{Fabrizio:Morro:1988}, \cite{Giorgi:Naso:Pata:2001}, \cite{Giorgi:Naso:Pata:2005}.

It is important to note that in \cite{Amar:Andreucci:Bisegna:Gianni:2009}--\cite{Amar:Andreucci:Bisegna:Gianni:2010}
our approach was based on eigenvalue estimates which
made it possible to keep into account (as far as the asymptotic rate of convergence is concerned)
both the dissipative properties of the intra/extra cellular phases and the dissipative properties of
the membranes.

Instead, in the nonlinear but coercive case, we proceed by exploiting the coercivity of $f$,
hence the electrical properties of the intra/extra cellular phases do not appear in the
rate of convergence.

If $f$ is not coercive it must be assumed to be
monotone increasing and we proceed via a Liapunov-style technique so that the rate of convergence is not
quantified.

\medskip

The paper is organized as follows: in Section \ref{s:prel} we present the
geometrical setting and the nonlinear differential model governing our problem at the microscale $\eps$.
In Section \ref{s:asymptotic1}
we prove the decay in time of the solution of the microscopic problem.
Finally, in Section \ref{s:asymptotic_hom} we prove the decay in time
of the solution of the macroscopic (or homogenized) problem, providing also the differential system
satisfied by such asymptotic limit.

\bigskip

\section{Preliminaries}
\label{s:prel}

Let $\Om$ be an open bounded subset of $\R^N$. In the sequel $\gamma$ or $\widetilde\gamma$ will
denote constants which may vary from line to line and which depend on the characteristic
parameters of the problem, but which are independent of the quantities tending to zero,
such as $\eps$, $\delta$ and so on, unless explicitly specified.

\subsection{The geometrical setting}\label{ss:geometry}
The typical geometry we
have in mind is depicted in Figure \ref{fig:omega}.
\begin{figure}[htbp]%
\begin{center}%
\input{fig_omega}%
\caption{\small On the left: an example of admissible periodic unit
cell $Y=\Perint\cup\Perout\cup\Permemb$ in $\R^{2}$. Here $\Perint$
is the shaded region and $\Permemb$ is its boundary. The remaining
part of $Y$ (the white region) is $\Perout$. On the right: the
corresponding domain $\Om=\Omint\cup\Omout\cup\Memb$. Here $\Omint$
is the shaded region and $\Memb$ is its boundary. The remaining part
of
$\Om$ (the white region) is $\Omout$.}%
\label{fig:omega}%
\end{center}%
\end{figure}%
In order to be more specific, assume $N\geq2$
and let us introduce a periodic open subset $\Per$ of $\RN$, so that
$\Per+z=\Per$ for all $z\in\ZZ^{N}$. For all $\eps>0$ define
$\Omint=\Om\cap\eps \Per$, $\Omout=\Om\setminus\overline{\eps \Per}$.
We assume that $\Om$, $\Per$ have regular boundary, say of class
${\mathcal C}^{\oo}$ for the sake of simplicity, and
$dist(\Memb,\partial\Om)\geq \gamma\eps$, where
$\Memb=\partial\Omint$.  We also employ the notation $Y=(0,1)^{N}$,
and $\Perint=\Per\cap Y$, $\Perout=Y\setminus\overline{\Per}$,
$\Permemb=\partial\Per\cap \overline Y$. As a simplifying assumption,
we stipulate that $\Perint$ is a connected smooth subset of $Y$ such that
$dist(\Permemb,\partial Y)>0$.  We denote by
$\nu$ the normal unit vector to $\Permemb$ pointing into $\Perout$, so
that $\nu_\eps(x)=\nu(\eps^{-1}x)$.

For later use, we introduce also the conductivity
\begin{equation*}
  \dfboth(y)=\begin{cases}\dfint & \text{if $y\in \Perint$,}\\ \dfout & \text{if $y\in\Perout$,}
  \end{cases}
  \qquad\text{and}\qquad
   \dfav=|\Perint|\dfint+|\Perout|\dfout\,,
\end{equation*}
where $\dfint,\dfout$ are positive constants, and we also set
$\dfbotheps(x)=\dfboth(\eps^{-1}x)$.
Moreover, let us set
\begin{equation*}
\Cper^k_\#(Y):=\{u :Y\setminus\Permemb\to\R\ |\
u_{\mid \Perint}\in {\mathcal C}^k(\overline\Perint)\,,
  \ u_{\mid \Perout}\in {\mathcal C}^k(\overline\Perout)\,,\ \text{and}\ u\ \text{is } Y-\text{periodic}\}\,,
\end{equation*}
for every $0\leq k\leq +\infty$, and
\begin{equation*}
{\X}^1_\#(Y):=\{u\in L^2(Y)\ |\
u_{\mid \Perint}\in H^1(\Perint)\,,
  \ u_{\mid \Perout}\in H^1(\Perout)\,,\ \text{and}\ u\ \text{is } Y-\text{periodic}\}\,.
\end{equation*}
More generally, the subscript $\#$ in the definition of a function space will denote periodicity with respect
to the first domain, in such a way that the extended function remains (locally) in the same space.

We set also
\begin{equation*}
{\X}^1(\Om_\eps):=\{u\in L^2(\Om)\ |\
u_{\mid \Omint}\in H^1(\Omint),\ u_{\mid \Omout}\in H^1(\Omout)\}\,.
\end{equation*}
We note that, if $u\in {\X}^1_\#(Y)$ then the traces of $u_{\mid \Per_i}$ on $\Permemb$, for $i=1,2$,
belong to $H^{1/2}(\Permemb)$, as well as
$u\in {\X}^1(\Om_\eps)$ implies that the traces of $u_{\mid \Om^\eps_i}$ on $\Memb$, for $i=1,2$,
belong to $H^{1/2}(\Memb)$.

\subsection{Statement of the problem}
\label{ss:statment}

We write down the model problem:
\begin{alignat}2
  \label{eq:PDEin}
  -\Div(\dfboth^\eps \nabla u_{\eps})&=0\,,&\qquad &\text{in $(\Omint\cup\Omout)\times(0,T)$;}\\
  \label{eq:FluxCont}
  [\dfbotheps \nabla u_\eps \cdot \nu_{\eps}] &= 0
  \,,&\qquad &\text{on
  $\Memb\times(0,T)$;}\\
  \label{eq:Circuit}
   \frac{\alpha}{\eps}\frac{\partial}{\partial t}[u_{\eps}] +
   f\left(\frac{[u_{\eps}]}{\eps}\right) &=\dfbotheps \nabla u_{\eps} \cdot \nu_{\eps}\,,&\qquad
  &\text{on $\Memb\times(0,T)$;}\\
  \label{eq:InitData}
  [u_{\eps}](x,0)& =S_{\eps}(x) \,,&\qquad
  &\text{on $\Memb$;}\\
  \label{eq:BoundData}
    u_{\eps}(x)&=\dato(x,t)\,,&\qquad&\text{on $\partial\Om\times(0,T)$,}
\end{alignat}
where $\dfboth^\eps$ is defined in the previous subsection and $\alpha>0$ is a constant.
We note that, by the definition already given in the previous section,
$\nu_{\eps}$ is the normal unit vector to $\Memb$ pointing into $\Omout$.
Since $u_{\eps}$ is not in general continuous across $\Memb$ we set
\begin{equation*}
  u_{\eps}^{(1)} \defeq \text{ trace of $u_{\eps\mid \Omint}$ on
  $\Memb\times(0,T)$;}\quad
  u_{\eps}^{(2)} \defeq \text{ trace of $u_{\eps\mid
  \Omout}$ on $\Memb\times(0,T)$.}
\end{equation*}
Indeed we refer conventionally to $\Omint$ as to the
\textit{interior domain},
and to $\Omout$ as to the \textit{outer domain}. We also denote
\begin{equation*}
  [u_{\eps}] \defeq u_{\eps}^{{(2)}} - u_{\eps}^{{(1)}}\,.
\end{equation*}
Similar conventions are employed for other quantities, for example in
\eqref{eq:FluxCont}.
In this framework we will assume that
\begin{equation}\label{eq:assumpt2}
i)\quad S_\eps\in H^{1/2}(\Memb)\,,\qquad\qquad
ii)\quad \int_{\Memb}S^2_\eps(x)\di\sigma \leq \gamma\eps\,,\\
\end{equation}
where the second assumption in \eqref{eq:assumpt2} is  needed in order that the solution of system \eqref{eq:PDEin}--\eqref{eq:BoundData} satisfies the classical energy inequality. (see (3.1) in
\cite{Amar:Andreucci:Gianni:2014a}).

Moreover, $f:\R\to\R$ satisfies
\begin{alignat}1
\label{eq:assumpt1}
& f\in \CC^1(\R)\,,
\\
\label{eq:a37}
& f\ \text{is a strictly monotone increasing function,}
\\
\label{eq:assumpt1bis}
& f(0)=0\,,\\
\label{eq:a38}
& f^{\prime }(s)\geq \delta_0\,,\qquad  \text{for a suitable $\delta_0>0$ and $\forall |s|$ sufficiently
large.}
\end{alignat}
The previous assumptions imply also
\begin{equation}
\label{eq:a36}
f(s)s \geq \lambda_1 s^{2}-\lambda_2|s|\,,\qquad \text{for some constants $\lambda_1>0$ and $\lambda_2\geq 0$.}
\end{equation}

Notice that the results presented in this paper hold also in a more general case, namely if we replace condition \eqref{eq:a38} with the assumption that $f^{-1}$ is uniformly continuous in $\R$, for example when $f(s)= s+\sin s$.

Finally, $\dato:\Om\times\R\to\R$ is a function satisfying the following assumptions
\begin{equation}\label{eq:h1}
\begin{aligned}
i)\ \ \quad & \dato\in L^2_{loc}\big(\R;H^{2}(\Om)\big)\,;\\
ii)\ \quad & \dato_t\in L^2_{loc}\big(\R;H^1(\Om)\big)\,;\\
iii)\quad & \text{$\dato(x,\cdot)$ is $1$-periodic}\quad & \text{for a.e. $x\in\Om$.}
\end{aligned}
\end{equation}
Existence and uniqueness for problem \eqref{eq:PDEin}--\eqref{eq:BoundData} has been proved in \cite{Amar:Andreucci:Bisegna:Gianni:2005}.
Moreover, by \cite[Lemma 4.1 and Remark 4.2]{Amar:Andreucci:Gianni:2014a} it follows that the solution
$u_\eps\in\CC^0\big((0,T];\X^1(\Om_\eps)\big)$ and $[u_\eps]\in \CC^0\big((0,T];L^2(\Memb)\big)$, uniformly
with respect to $\eps$, and $[u_\eps]\in \CC^0\big([0,T];L^2(\Memb)\big)$, but with non uniform estimates.

\section{Asymptotic convergence to a periodic solution of the $\eps $-problem}\label{s:asymptotic1}
The purpose of this section is to prove the asymptotic convergence of the
solution of problem \eqref{eq:PDEin}--\eqref{eq:BoundData} to a periodic function
$u^{\#}_\eps$ when $t\to+\infty $.
The function $u_{\eps}^{\#}$ is, in turn, a solution of the system
\begin{alignat}2
    \label{eq:per_PDEboth}
    -\Div(\dfbotheps \nabla u_{\eps}^{\#})&=0\,,&\quad &\text{in $(\Omint\cup\Omout)\times\R$;}\\
    \label{eq:per_FluxCont}
    [\dfbotheps \nabla u_{\eps}^{\#} \cdot \nu_\eps] &= 0 \,,&\ &\text{on $\Memb\times\R$;}\\
    \label{eq:per_Circuit}
    \frac{\alpha}{\eps} \pder{}{t}[u_{\eps}^{\#}] + f\left(\frac{[u_\eps^{\#}]}{\eps}\right)
    &=(\dfbotheps \nabla u_{\eps}^{\#} \cdot \nu_\eps)\,,&\ &\text{on $\Memb\times\R$;}\\
    \label{eq:per_BoundData}
    u_{\eps}^{\#}(x,t)&=\dato(x,t)\,,&\ &\text{on $\partial\Om\times\R$;}\\
    \label{eq:per_periodicity}
    u_{\eps}^{\#}(x,\cdot)&\text{ is $1$-periodic,}&\ &\text{in $\Om$.}
\end{alignat}
Indeed, this problem is derived from
\eqref{eq:PDEin}--\eqref{eq:BoundData} replacing equation
\eqref{eq:InitData} with \eqref{eq:per_periodicity}.

The rigorous definition of weak solution of \eqref{eq:per_PDEboth}--\eqref{eq:per_periodicity} is standard
(see for instance \cite[Definition 4.13]{Amar:Andreucci:Gianni:2014a}

As a first step we will prove the following result.

\begin{prop}\label{p:prop6}
Under the assumptions \eqref{eq:assumpt1}--\eqref{eq:a38} and \eqref{eq:h1},
problem \eqref{eq:per_PDEboth}--\eqref{eq:per_periodicity} admits a solution $u^\#_\eps\in \spaziobbbdiesis$.
\end{prop}
\begin{proof}
For $\delta >0$, let us denote by $f_{\delta}(s):= f(s)+\delta s$, for every $s\in\R$,
and consider the problem
\begin{alignat}2
    \label{eq:per_PDEbothdelta}
    -\Div(\dfbotheps \nabla u_{\eps,\delta}^{\#})&=0\,,&\ &\text{in $(\Omint\cup\Omout)\times\R$;}\\
    \label{eq:per_FluxContdelta}
    [\dfbotheps \nabla u_{\eps,\delta}^{\#} \cdot \nu_\eps] &= 0 \,,&\ &\text{on $\Memb\times\R$;}\\
    \label{eq:per_Circuitdelta}
    \frac{\alpha}{\eps} \pder{}{t}[u_{\eps,\delta}^{\#}] + f_\delta\left(\frac{[u_{\eps,\delta}^{\#}]}{\eps}\right)
    &=\dfbotheps \nabla u_{\eps,\delta}^{\#} \cdot \nu_\eps\,,&\quad &\text{on $\Memb\times\R$;}\\
    \label{eq:per_BoundDatadelta}
    u_{\eps,\delta}^{\#}(x,t)&=\dato(x,t)\,,&\ &\text{on $\partial\Om\times\R$;}\\
    \label{eq:per_periodicitydelta}
    u_{\eps,\delta}^{\#}(x,\cdot)&\text{ is $1$-periodic,}&\ &\text{in $\Om$.}
\end{alignat}
For any positive $\eps$ and $\delta $, the previous problem admits a unique time-periodic
solution because of the results already proved in \cite{Amar:Andreucci:Gianni:2014a}.

On the other hand, multiplying
equation \eqref{eq:per_PDEbothdelta} by $u^\#_{\eps,\delta}-\dato $, integrating by parts on
$\Om\times[ 0,1]$, using the periodicity and taking into account equations
\eqref{eq:per_FluxContdelta}--\eqref{eq:per_BoundDatadelta}, we get
\begin{equation}\label{eq:a39}
\int_0^1\int_{\Om}\frac{\dfbotheps}{2}|\nabla u_{\eps,\delta}^{\#}|^{2}\di x\di t+
\int_0^1\int_{\Memb}f_\delta\left(\frac{[ u_{\eps,\delta}^{\#}] }{\eps}\right)[ u_{\eps,\delta }^{\#}] \di\sigma \di t
\leq \int_0^1\int_{\Om}\frac{\dfbotheps}{2}| \nabla \dato |^{2}\di x\di t\,.
\end{equation}
Finally, using \eqref{eq:a36} we obtain
\begin{equation}\label{eq:a40}
\int_0^1\!\!\int_{\Om}\frac{\dfbotheps}{2}|\nabla u_{\eps,\delta}^{\#}|^{2}\di x\di t
+\int_0^1\!\!\int_{\Memb}\frac{\lambda_1}{2\eps}[ u_{\eps,\delta}^{\#}]^{2}\di\sigma \di t
\leq \int_0^1\!\!\int_{\Om}\frac{\dfbotheps}{2}|\nabla \dato|^{2}\di x\di t
+\frac{\eps}{2\lambda_1}\lambda_2^{2}|\Memb|\,.
\end{equation}
Multiplying now equation \eqref{eq:per_PDEbothdelta} by $u^{\#}_{\eps,\delta,t }-\dato_t$,
integrating by parts on $\Om\times [0,1]$, using the periodicity and taking into account
equations \eqref{eq:per_FluxContdelta}--\eqref{eq:per_BoundDatadelta}, we get
\begin{multline}\label{eq:a41}
\frac{\alpha }{\eps}\int_0^1\!\!\int_{\Memb}[ u_{\eps,\delta,t}^{\#}]^{2}\di\sigma \di t+
\int_0^1\!\!\int_{\Memb}f_{\delta }\left(\frac{[u_{\eps,\delta }^{\#}]}{\eps}\right)
[u_{\eps,\delta,t}^{\#}] \di \sigma\di t
\\
\leq \int_0^1\!\!\int_\Om \dfbotheps\nabla u_{\eps,\delta}^{\#}\nabla \dato_t\di x\di t
\leq \int_0^1\!\!\int_\Om \frac{\dfbotheps}{2}|\nabla u_{\eps,\delta}^{\#}|^2\di x\di t
+\int_0^1\!\!\int_\Om \frac{\dfbotheps}{2}|\nabla \dato_t|^2\di x\di t
\\
\leq \int_0^1\!\!\int_{\Om}\frac{\dfbotheps}{2}|\nabla \dato|^{2}\di x\di t
+\frac{\eps}{2\lambda_1}\lambda_2^{2}|\Memb|
+\int_0^1\!\!\int_\Om \frac{\dfbotheps}{2}|\nabla \dato_t|^2\di x\di t\,,
\end{multline}
where we used \eqref{eq:a40}. Notice that the second integral on the left-hand side is equal to
zero by periodicity and trivial integration.
Hence
\begin{equation}\label{eq:a42}
\frac{\alpha }{2\eps}\int_0^1\int_{\Memb}[ u_{\eps,\delta,t}^{\#}]^{2}\di\sigma \di t
\leq \int_0^1\int_{\Om}\frac{\dfbotheps}{2}| \nabla \dato_{t}|^{2}\di x\di t
+\int_0^1\int_{\Om}\frac{\dfbotheps}{2}| \nabla \dato |^{2}\di x\di t
+\frac{\eps}{2\lambda_1}\lambda_2^{2}|\Memb|\,.
\end{equation}
Inequalities \eqref{eq:a40} and \eqref{eq:a42},
for $\eps>0$ fixed,
yield the weak convergence of $u_{\eps,\delta}^{\#}$ and $\nabla u_{\eps,\delta}^{\#}$ in $L^{2}(\Om^\eps_i\times(0,1))$, $i=1,2$,
and respectively the strong convergence of $[ u_{\eps,\delta}^{\#}]$ in $L^{2}(\Memb\times(0,1))$,
for $\delta\to 0$.
Since all the functions $u_{\eps,\delta}^{\#}$ are $1$-periodic,
denoting as usual with $u_{\eps}^{\#}$ the limit of $u_{\eps,\delta}^{\#}$ we have that the
same periodicity holds true for $u_{\eps}^{\#}$.
Moreover we can pass to the limit, as $\delta
\rightarrow 0$, in the weak formulation of problem \eqref{eq:per_PDEbothdelta}--\eqref{eq:per_periodicitydelta},
thus obtaining that $u_{\eps}^{\#}$ is a
$1$-periodic solution of problem \eqref{eq:per_PDEboth}--\eqref{eq:per_periodicity},
under the assumptions \eqref{eq:assumpt1}--\eqref{eq:a38} and \eqref{eq:h1}.

Differentiating formally with respect to $t$ \eqref{eq:per_PDEbothdelta}--\eqref{eq:per_BoundDatadelta},
multiplying the first equation thus obtained by $(u_{\eps,\delta,t}^{\#}-\Psi_t)$ and finally integrating by parts, we obtain
\begin{equation*}
\int_0^1\int_\Om \frac{\dfboth^\eps}{2}|\nabla u_{\eps,\delta,t}^{\#}|^2\di x\di t\leq
\int_0^1\int_\Om \frac{\dfboth^\eps}{2}|\nabla \Psi_{t}|^2\di x\di t\,.
\end{equation*}

Since the estimates above are uniform in $\delta$, we have that $u^\#_\eps$ belongs to the class claimed in
the statement.
\end{proof}


Given $\eps>0$, it remains to prove the asymptotic convergence of the solution $u_\eps$ of \eqref{eq:PDEin}--\eqref{eq:BoundData}
to\ $u_\eps^{\#}$, for $t\to +\infty$.

\begin{thm}\label{t:t5}
Let $\eps>0$ be fixed and let $u_\eps$ be the solution of problem
\eqref{eq:PDEin}--\eqref{eq:BoundData}. Then, for $t\to +\infty$, $u_\eps\to u^{\#}_\eps$
in the following sense:
\begin{alignat}1
\label{eq:decayperiodic_new}
\qquad & \lim_{t\to+\infty}\| u_\eps(\cdot,t)-u^{\#}_\eps(\cdot,t)\| _{L^{2}(\Om)}=0\,;
\\
\label{eq:decayperiodic1_new}
\qquad & \lim_{t\to+\infty}\| \nabla u_\eps(\cdot,t)-
\nabla u^{\#}_\eps(\cdot,t)\| _{L^{2}(\Om)}=0\,;
\\
\label{eq:decayperiodic3_new}
\qquad & \lim_{t\to+\infty}\|[u_\eps](\cdot,t)-[u^{\#}_\eps](\cdot,t)\|_{L^{2}(\Memb)}=0\,.
\end{alignat}
\end{thm}

\begin{proof}
Setting $r_\eps:=u_\eps^{\#}-u_\eps$, we obtain that $r_\eps$ satisfies
\begin{alignat}2
  \label{eq:PDEboth1r}
  -\Div(\dfbotheps\nabla r_{\eps})&=0\,,&\qquad &\!\!\!\!\!\!\!\!\!\!\!\!\!\!\!\!\!\!
  \text{in $(\Omint\cup\Omout)\times(0,+\infty)$;}\\
  \label{eq:FluxCont1r}
  [\dfbotheps\nabla r_\eps\cdot \nu_{\eps}] &= 0
  \,,&\qquad &\text{on   $\Memb\times (0,+\infty)$;}\\
  \label{eq:Circuit1r}
   \frac{\alpha}{\eps}\frac{\partial}{\partial t}[r_{\eps}] +
   g_\eps(x,t)\frac{[r_{\eps}]}{\eps}&=\dfbotheps \nabla r_{\eps} \cdot \nu_{\eps}\,,&\qquad
  &\text{on $\Memb\times(0,+\infty)$;}\\
  \label{eq:InitData1r}
  [r_{\eps}](x,0)& =[u_\eps^{\#}(x,0)]-S_{\eps}(x)=:\widehat{S_\eps}(x) \,,&\qquad
  &\text{on $\Memb$;}\\
  \label{eq:BoundData1r}
  r_{\eps}(x)&=0 \,,&\qquad&\text{on $\partial\Om\times(0,+\infty)$;}
\end{alignat}
where
\begin{equation*}
    g_\eps(x,t):=
    \left\{
\begin{aligned}
    & f^\prime\left(\frac{[ u_\eps]}{\eps} (x,t)\right)
 \qquad & \text{if $ [ u_\eps](x,t)=[ u_\eps^\#](x,t)$,}
     \\
  &  \frac{f\left(\frac{[ u_\eps^{\#}]}{\eps }(x,t)\right)
      -f\left(\frac{[ u_\eps]}{\eps} (x,t)\right)}{\frac{[ u_\eps^{\#}]}{\eps} (x,t)
      -\frac{[ u_\eps]}{\eps}(x,t)}
      \qquad &  \text{if $ [ u_\eps](x,t)\not=[ u_\eps^\#](x,t)$,}
    \end{aligned}\right.
  \end{equation*}
so that $g_\eps(x,t)\geq 0$, and $\widehat{S_\eps}(x)$ still
satisfies assumption $ii)$ in \eqref{eq:assumpt2} because of \eqref{eq:a40} and \eqref{eq:a42}.
Multiplying equation \eqref{eq:PDEboth1r} by $r_\eps $ and integrating by parts we have
\begin{equation}\label{eq:a43}
\int_{\Om}\dfbotheps| \nabla r_\eps|^{2}\di x+
\frac{\alpha }{\eps}\int_{\Memb}[r_{\eps,t}][ r_{\eps}] \di\sigma
+\int_{\Memb}\frac{g_\eps(x,t)}{\eps}[r_\eps]^{2}\di\sigma =0\,.
\end{equation}
Equation \eqref{eq:a43} implies that the function $t\mapsto
\frac{\alpha }{\epsilon }\int_{\Memb}[r_\eps(x,t)]^{2}\di\sigma $
is a positive, decreasing function of $t$; hence, it
tends to a limit value $\overline r_\eps\geq 0$ as $t\to +\infty $.
We claim that the value $\overline r_\eps$ must be zero. Otherwise,
for every $t>0$, $\frac{\alpha }{\epsilon }\int_{\Memb}[r_\eps(x,t)]^{2}\di\sigma\geq
\overline r_\eps>0$. On the other hand, setting
$\Memb_{\overline r_\eps}(t):=\{x\in \Memb :[r_\eps(x,t)] ^{2}\leq \frac{\overline r_\eps\eps}{2\alpha|\Memb|}\}$,
we have that
\begin{equation}\label{eq:a44}
\frac{\alpha }{\eps }\int_{\Memb\setminus\Memb_{\overline r_\eps}(t)}[r_\eps(x,t)]^{2}\di\sigma
\geq \frac{\overline r_\eps}{2}\,,\qquad \forall t>0\,.
\end{equation}
Indeed, by definition,
\begin{equation*}
\begin{aligned}
\overline r_\eps\ & \leq \frac{\alpha}{\eps}\int_{\Memb} [r_\eps(x,t)]^2\di \sigma
= \frac{\alpha}{\eps}\int_{\Memb\setminus\Memb_{\overline r_\eps}(t)}[r_\eps(x,t)]^2\di\sigma+
\frac{\alpha}{\eps}\int_{\Memb_{\overline r_\eps}(t)}[r_\eps(x,t)]^2\di\sigma
\\
& \leq \frac{\alpha}{\eps}\int_{\Memb\setminus\Memb_{\overline r_\eps}(t)}[r_\eps(x,t)]^2\di\sigma+
\frac{\alpha}{\eps}\,\frac{\overline r_\eps\eps}{2\alpha|\Memb|}\,|\Memb_{\overline r_\eps}|
\leq \frac{\alpha}{\eps}\int_{\Memb\setminus\Memb_{\overline r_\eps}(t)}[r_\eps(x,t)]^2\di\sigma+
\frac{\overline r_\eps}{2}\,,
\end{aligned}
\end{equation*}
which implies \eqref{eq:a44}.
Moreover,  we have that, on $\Memb\setminus\Memb_{\overline r_\eps}(t)$,
$g_\eps(x,t)\geq \chi> 0$, where $\chi $
is a suitable positive constant depending only on  $(\overline r_\eps,\eps,\alpha ,|\Memb|)$ (this last result
follows from assumption \eqref{eq:a37}--\eqref{eq:a38}).
Hence, using \eqref{eq:a43}, it follows
\begin{multline}\label{eq:a45}
\frac{d}{dt}\left( \frac{\alpha }{2\eps}\int_{\Memb}[r_\eps(x,t)]^{2}\di\sigma \right)
\leq -\!\!\!\!\int_{\Memb\setminus\Memb_{\overline r_\eps}(t)}\frac{g_\eps(x,t)}{\eps}[r_\eps(x,t)]^{2}\di\sigma
\\
\leq -\chi \int_{\Memb\setminus\Memb_{\overline r_\eps}(t)}\frac{1}{\eps}
[r_\eps(x,t)]^{2}\di\sigma
\leq -\frac{\overline r_\eps}{2\alpha }\chi<0\,.
\end{multline}
Inequality \eqref{eq:a45} clearly contradicts the asymptotic convergence in $t$ of the function
$t\mapsto\frac{\alpha }{\eps }\int_{\Memb}[r_\eps(x,t)]^{2}\di\sigma$, hence
\begin{equation}\label{eq:a46}
\lim_{t\to +\infty }\frac{\alpha }{\eps}\int_{\Memb}[r_\eps(x,t)]^{2}\di\sigma =0\,.
\end{equation}
In particular, this gives \eqref{eq:decayperiodic3_new}.
Integrating \eqref{eq:a43} in $[\hat t,\infty )$ and taking into account \eqref{eq:a46}, we get
\begin{equation}\label{eq:a47}
\int_{\hat t}^{+\infty }\int_{\Om}\dfbotheps|\nabla r_\eps|^{2}\di x\di t\leq
\frac{\alpha }{2\eps}\int_{\Memb}[r_\eps(x,\hat t)]^{2}\di\sigma\,,
\end{equation}
which implies
\begin{equation}\label{eq:a48}
\lim_{t\to +\infty }\int_{\hat t}^{+\infty }\int_{\Om}\dfbotheps|\nabla r_\eps|^{2}\di x\di t=0\,.
\end{equation}
Condition \eqref{eq:a48} guarantees that for every positive $\eta$ there exists a $\hat{t}(\eta)>0$,
such that
\begin{equation*}
\int_{\hat{t}(\eta)}^{+\infty}\int_{\Om}\dfbotheps| \nabla r_\eps|^{2}\di x\di t
\leq \eta\,,
\end{equation*}
which, in turn implies that, for every natural number $n$, there exists a
$t_{n}\in (\hat{t}(\eta)+n,\hat{t}(\eta)+(n+1))$, such that
\begin{equation}\label{eq:a49}
\int_{\Om}\dfbotheps|\nabla r_\eps(x,t_n)|^{2}\di x\leq \eta\,.
\end{equation}
Now, we multiply \eqref{eq:PDEboth1r} by $r_{\eps,t}$ and integrate in $\Om$, so that
\begin{equation}\label{eq:a50}
\int_{\Om}\dfbotheps\nabla r_\eps\nabla r_{\eps,t}(x,t)\di x
+\frac{\alpha }{\eps}\int_{\Memb}[ r_{\eps,t}(x,t)]^{2}\di\sigma
+\int_{\Memb}\frac{g_\eps(x,t)}{\eps}[r_\eps(x,t)]\,[r_{\eps,t}(x,t)] \di\sigma =0\,,
\end{equation}
which implies
\begin{equation}\label{eq:a51}
\int_{\Om}\dfbotheps\nabla r_\eps\nabla r_{\eps,t}\di x
\leq \int_{\Memb}\frac{g_\eps^{2}(x,t)}{2\alpha \eps}[r_\eps]^{2}\di\sigma\,.
\end{equation}
Moreover, integrating \eqref{eq:a51} in $[t_n,t^*]$ with $t^*\in [t_{n},t_{n}+2]$ and
using \eqref{eq:a49}, we have
\begin{equation*}
\sup_{t\in [t_{n},t_{n}+2]}\left( \int_{\Om}\frac{\dfbotheps}{2}| \nabla r_\eps(x,t)|^{2}\di x\right)
\leq \frac{\eta}{2}+\frac{L^{2}}{\alpha^{2}}\sup_{t\in [t_{n},+\infty )}
\left( \frac{\alpha }{\eps}\int_{\Memb}[r_\eps(x,t)]^{2}\di\sigma \right)\,,\quad
\forall n\in\NN\,.
\end{equation*}
Since $t_{n+1}-t_n<2$, the intervals of the form $[t_n,t_n+2]$, when $n$ varies in $\NN$,
are overlapping; hence, we obtain
\begin{equation}\label{eq:a52}
\sup_{t\in [\hat{t}+1,+\infty )}\left( \int_{\Om}\frac{\dfbotheps}{2}| \nabla r_\eps(x,t)|^{2}\di x\right)
\leq \frac{\eta}{2}+\frac{L^{2}}{\alpha^{2}}\sup_{t\in [\hat{t},+\infty)}
\left( \frac{\alpha }{\eps}\int_{\Memb}[r_\eps(x,t)]^{2}\di\sigma \right)\,.
\end{equation}
Because of \eqref{eq:a46} the integral in the right-hand side of \eqref{eq:a52} can be made
smaller than $\frac{\eta}{2}\left( \frac{L^{2}}{\alpha^{2}}\right)^{-1}$,
provided $\hat{t}$ is chosen sufficiently large in
dependence of $\eta$. This means that
\begin{equation}\label{eq:a53}
\sup_{t\in [\hat{t}+1,+\infty )}\left( \int_{\Om}\frac{\dfbotheps}{2}|\nabla r_\eps(x,t)|^{2}\di x\right)
\leq \eta\,,
\end{equation}
so that
\begin{equation}\label{eq:a54}
\lim_{t\to +\infty }\int_{\Om}\dfbotheps|\nabla r_\eps(x,t)|^{2}\di x=0\,.
\end{equation}
In particular, this gives \eqref{eq:decayperiodic1_new}.
Finally, Poincare's inequality together with \eqref{eq:a46} and \eqref{eq:a54} yield
\begin{equation}\label{eq:a55}
\lim_{t\to +\infty }\int_{\Om}|r_\eps(x,t)|^{2}\di x=0\,,
\end{equation}
which gives \eqref{eq:decayperiodic_new}.
\end{proof}

\begin{remark}\label{r:rem10tris}
More in general, the previous procedure allows us to prove that solutions of \eqref{eq:PDEin}--\eqref{eq:BoundData}
having different initial data satisfying \eqref{eq:assumpt2} but the same boundary condition tend asymptotically one to the other (such convergence  being exponential if $f$ is coercive in the sense of \eqref{eq:intro1}).
\end{remark}

\begin{remark}\label{r:rem10bis}
Observe that, thanks to previous remark, Theorem \ref{t:t5} implies uniqueness of the
periodic solution of problem \eqref{eq:per_PDEboth}--\eqref{eq:per_periodicity} in ${\spaziobbbdiesis}$.
\end{remark}

\section{Asymptotic decay of the solution of the homogenized problem}\label{s:asymptotic_hom}

The aim of this section is to prove asymptotic decay of the solution of the homogenized problem.
To this purpose,
let $(u,u^1)\!\in\! L^2\!\big(\!(0,T);$ $H^1(\Om)\big)\times L^2\big(\Om\times(0,T);{\X}^1_\#(Y)\big)$ be
the two-scale limit of the solution $u_\eps$ of problem \eqref{eq:PDEin}--\eqref{eq:BoundData},
where the initial data $S_\eps$ satisfies the
additional condition that
$S_{\eps}/\eps$ two-scale converges in $L^2\big(\Om;L^2(\Permemb)\big)$
 to a function $S_{1}$ such that $S_1(x,\cdot)=S_{\mid\Permemb}(x,\cdot)$
 for some $S\in{\mathcal C}\big(\overline\Om;{\mathcal C}^1_\#(Y)\big)$, and
\begin{equation}
\label{eq:init_asym1}
\lim_{\eps\to 0}\eps\int_{\Memb} \left(\frac{S_\eps}{\eps}\right)^2\!\!\!(x)\di\sigma
=\int_\Om\int_{\Permemb} S_1^2(x,y)\di x\di\sigma\,.
\end{equation}
We recall that, under these assumptions, by
\cite[Theorem 2.1]{Amar:Andreucci:Bisegna:Gianni:2013}, the pair $(u,u^1)$ is the weak solution
of the two-scale problem
 \begin{alignat}2
  \label{eq:PDE_limit}
  -\Div\left(\dfav \nabla u+\int_Y\dfboth\nabla_yu^{1}\di y\right)&=0\,,&\qquad
  &\text{in $\Om\times(0,T)$;}\\
  \label{eq:PDEper_limit}
  -\Div_y(\dfboth\nabla u+\dfboth \nabla_y u^{1})&=0\,,&\qquad &\!\!\!\!\!\!\!\!\!\!\!\!\!\!\!\!\!\!\!
  \text{in $\Om\times(\Perint\cup\Perout)\times(0,T)$;}\\
  \label{eq:FluxCont_limit}
  [\dfboth (\nabla u+\nabla_y u^{1}) \cdot \nu] &=0\,,&\qquad
  &\text{on $\Om\times\Permemb\times(0,T)$;}\\
  \label{eq:Circuit_limit}
   {\alpha}\frac{\partial}{\partial t}[u^{1}] +
   f\left({[u^{1}]}\right) &=\dfboth (\nabla u+\nabla_y u^{1}) \cdot \nu\,,&\qquad
  &\text{on $\Om\times\Permemb\times(0,T)$;}\\
  \label{eq:InitData_limit}
  [u^{1}](x,y,0)& =S_{1}(x,y) \,,&\qquad
  &\text{on $\Om\times\Permemb$;}\\
  \label{eq:BoundData_limit}
  u(x,t)&=\dato(x,t)\,,&\qquad&\text{on $\partial\Om\times(0,T)$;}
\end{alignat}
in the sense of the following definition.

\begin{defin}\label{d:def1}
A pair $(u,u^1)\!\in\!
  L^2\!\big(\!(0,T);$ $H^1(\Om)\big)\times
  L^2\big(\Om\times(0,T);{\X}^1_\#(Y)\big)$ is a weak solution of
  \eqref{eq:PDE_limit}--\eqref{eq:BoundData_limit} if
%
\begin{multline}\label{eq:a76}
\int_{0}^{T}\int_{\Om }\int_{Y}\dfboth \left( \nabla u+\nabla_{y}u^{1}\right)
\left(\nabla \phi +\nabla _{y}\Phi \right)\di x\di y\di t+
\int_{0}^{T}\int_{\Om}\int_{\Permemb}f([u^{1}])[\Phi]\di x\di\sigma\di t
\\
-\alpha \int_{0}^{T}\int_{\Om}\int_{\Permemb}[u^{1}]\frac{\partial }{\partial t}[\Phi ]\di x\di\sigma\di t
-\alpha \int_{\Om}\int_{\Permemb}[\Phi ]S_{1}\di x\di\sigma=0\,,
\end{multline}
for any function $\phi\in \CC^0\big(0,T;H^1_0(\Om)\big)$
and any function $\Phi\in \CC^0\big([0,T];L^2(\Om;{\X}^1_\#(Y))\big) $
  with $[\Phi_t]\in \CC^0\big([0,T];L^2(\Om\times\Permemb)\big)$
  which vanishes at $t=T$.

Moreover, $u$ satisfies the boundary condition on $\partial\Om\times[0,T]$
in the trace sense (i.e. $u(x,t)=\dato (x,t)$ a.e. on $\partial\Om\times(0,T)$) and $u^{1}$ is periodic in $Y$
and has zero mean value in $Y$ for a.e. $(x,t)\in\Om\times(0,T)$
(see \cite[Definition 5.1]{Amar:Andreucci:Gianni:2014a}).
\end{defin}
\bigskip

For later use, let us define
\begin{multline}\label{eq:a86}
\normaa{\big(h(\cdot,t),h^1(\cdot,t)\big)\big)}
:= \|h\|_{\CC^0([0,1];H^1(\Om))}
+\|h^{1}\|_{\spaziob}
\\
+\| \nabla_{y}h^{1}\|_{\spaziob}+\Vert [ h^1]\Vert_{\spazioc} \,,
\end{multline}
where $(h,h^1)\in \spazioaaa$, and
\begin{equation}\label{eq:a85}
\normab{(\widetilde h,\widetilde h^1)}
:= \Vert \widetilde h\Vert_{H^1(\Om)} + \Vert \widetilde h^1\Vert_{L^2(\Om\times Y) }+
\Vert \nabla_y \widetilde h^1\Vert_{L^2(\Om\times Y)}+\Vert [\widetilde h^1]\Vert_{L^2(\Om\times \Permemb)} \,,
\end{equation}
where $(\widetilde h,\widetilde h^1)\in \spazioee$.

As in the previous section, first we prove that there exists a time-periodic weak solution of the two-scale problem
\begin{alignat}2
  \label{eq:PDE_limit_per}
  -\Div\left(\dfav \nabla u^{\#}+\int_Y\dfboth\nabla_yu^{1,\#}\di y\right)&=0\,,&
  &\text{in $\Om\times\R$;}\\
  \label{eq:PDEper_limit_per}
  -\Div_y(\dfboth\nabla u^{\#}+\dfboth \nabla_y u^{1,\#})&=0\,,&
  &\!\!\!\!\!\!\!\!\!\!\!\!\!\!\!\!\!\!\!\text{in $\Om\times(\Perint\cup\Perout)\times\R$;}\\
  \label{eq:FluxCont_limit_per}
  [\dfboth (\nabla u^{\#}+\nabla_y u^{1,\#}) \cdot \nu] &=0\,,&
  &\text{on $\Om\times\Permemb\times\R$;}\\
  \label{eq:Circuit_limit_per}
   {\alpha}\frac{\partial}{\partial t}[u^{1,\#}] +
   f\left({[u^{1,\#}]}\right) &=\dfboth (\nabla u^{\#}+\nabla_y u^{1,\#}) \cdot \nu\,,&
  &\text{on $\Om\times\Permemb\times\R$;}\\
  \label{eq:InitData_limit_per}
  [u^{1,\#}](x,y,\cdot)& \text{ is $1$-periodic,}&   &\text{on $\Om\times\Permemb$;}\\
  \label{eq:BoundData_limit_per}
  u^{\#}(x,t)&=\dato(x,t)\,,& &\text{on $\partial\Om\times\R$;}
\end{alignat}
in the sense of the following definition.

\begin{defin}\label{d:def2}
A pair $(\vi^{\#},\vi^{1,\#})\in\spaziogperaa$ with
$[\vi^{1,\#}_t]\in L^2_\#\big(0,1;L^2(\Om\times\Permemb)\big)$
is a time-periodic weak solution (with period $1$) of \eqref{eq:PDE_limit_per}--\eqref{eq:BoundData_limit_per} if
  \begin{multline}\label{eq:a35_nuova}
    \int_{\R}\int_{\Om}\int_{Y}\dfboth
    \left( \nabla \vi^{\#}(x,t)+\nabla _{y}\vi^{1,\#}(x,y,t)\right)
    \left( \nabla \phi (x,t)+\nabla_{y}\Phi (x,y,t)\right) \di x\di y\di t
    \\
    +\int_{\R}\int_{\Om}\int_{\Permemb}
    f([\vi^{1,\#}(x,y,t)])[\Phi(x,y,t)]\di x\di\sigma\di t
    \\
    -\alpha \int_{\R}\int_{\Om}\int_{\Permemb}[\vi^{1,\#}(x,y,t)]
    \frac{\partial }{\partial t}[\Phi (x,y,t)]\di x\di\sigma\di t=0\,
  \end{multline}
for every $(\phi,\Phi)\in\spaziogzero$, $[\Phi_t]\in L^2\big(\R;L^2(\Om\times\Permemb)\big)$
and $\vi^{1,\#}$ has zero mean value in $Y$ for a.e. $(x,t)\in\Om\times\R$ and $\vi^\#$ satisfies
\eqref{eq:BoundData_limit_per} in the trace sense
(see \cite[Definition 5.7]{Amar:Andreucci:Gianni:2014a}).
\end{defin}

\begin{remark}\label{r:rem12}
We note that by a standard approximation of periodic testing functions with functions compactly supported in a period,
the weak formulation \eqref{eq:a35_nuova} can be equivalently rewritten as
%
  \begin{multline*}
    \int_0^1\int_{\Om}\int_{Y}\dfboth
    \left( \nabla \vi^{\#}(x,t)+\nabla _{y}\vi^{1,\#}(x,y,t)\right)
    \left( \nabla \phi (x,t)+\nabla_{y}\Phi (x,y,t)\right) \di x\di y\di t
    \\
    +\int_0^1\int_{\Om}\int_{\Permemb}
    f([\vi^{1,\#}(x,y,t)])[\Phi(x,y,t)]\di x\di\sigma\di t
    \\
    -\alpha \int_0^1\int_{\Om}\int_{\Permemb}[\vi^{1,\#}(x,y,t)]
    \frac{\partial }{\partial t}[\Phi (x,y,t)]\di x\di\sigma\di t=0\,
  \end{multline*}
for every $(\phi,\Phi)\in\spaziogperoaa$, $[\Phi_t]\in L^2_{\#}\big(0,1;L^2(\Om\times\Permemb)\big)$.
Hence, when it is more convenient, we replace
compactly supported testing functions with $1$-periodic testing functions.
\end{remark}

\begin{prop}\label{p:prop7}
Under the assumptions \eqref{eq:assumpt1}--\eqref{eq:a38} and \eqref{eq:h1},
problem \eqref{eq:PDE_limit_per}--\eqref{eq:BoundData_limit_per} admits a $1$-periodic in time
solution.
\end{prop}

\begin{proof}
For $\delta >0$, let us denote by $f_{\delta}(s):= f(s)+\delta s$, for every $s\in\R$,
and consider the problem
 \begin{alignat}2
  \label{eq:PDE_limit_perdelta}
  -\Div\left(\dfav \nabla u_\delta^{\#}+\int_Y\dfboth\nabla_yu_\delta^{1,\#}\di y\right)&=0\,,& \quad
  \text{in }\Om &\times\R\,;\\
  \label{eq:PDEper_limit_perdelta}
  -\Div_y(\dfboth\nabla u_\delta^{\#}+\dfboth \nabla_y u_\delta^{1,\#})&=0\,,&
  \!\!\!\!\!\!\!\!\!\!\!\!\!\!\!\!\!\!\!\text{in }\Om\times(\Perint\cup\Perout) &\times\R\,;\\
  \label{eq:FluxCont_limit_perdelta}
  [\dfboth (\nabla u_\delta^{\#}+\nabla_y u_\delta^{1,\#})\!\! \cdot\!\! \nu] &=0,& \quad
  \text{on }\Om\times\Permemb &\times\R\,;\\
  \label{eq:Circuit_limit_perdelta}
   {\alpha}\frac{\partial}{\partial t}[u_\delta^{1,\#}] +
   f_\delta\left({[u_\delta^{1,\#}]}\right) &=\dfboth (\nabla u_\delta^{\#}+\nabla_y u_\delta^{1,\#}) \cdot \nu\,,&
  \text{on }\Om\times\Permemb &\times\R\,;\\
  \label{eq:InitData_limit_perdelta}
  [u_\delta^{1,\#}](x,y,\cdot)& \text{ is $1$-periodic,}&   \text{on }\Om & \times\Permemb\,;\\
  \label{eq:BoundData_limit_perdelta}
  u_\delta^{\#}(x,t)&=\dato(x,t)\,,& \text{on }\partial\Om &\times\R\,;
\end{alignat}
where $u_\delta^{1,\#}$ has zero mean value on $Y$ for a.e. $(x,t)\in\Om\times\R$.

Since $f_\delta$ has a strictly positive derivative on $\R$, by the results proved in \cite[Section 5]{Amar:Andreucci:Gianni:2014a}, a unique periodic solution
$(u_{\delta }^{\#},u_\delta^{1,\#})$ of problem \eqref{eq:PDE_limit_perdelta}-- \eqref{eq:BoundData_limit_perdelta}
does exist, i.e. $(u_{\delta }^{\#},u_\delta^{1,\#})$ satisfies
\begin{multline}\label{eq:a35}
\int_{0}^{1}\int_{\Om}\int_{Y}\dfboth
\left( \nabla u_{\delta }^{\#}(x,t)+\nabla _{y}u_{\delta }^{1,\#}(x,y,t)\right)
\left( \nabla \phi (x,t)+\nabla_{y}\Phi (x,y,t)\right) \di x\di y\di t
\\
+\int_{0}^{1}\int_{\Om}\int_{\Permemb}
f_\delta([u_{\delta }^{1,\#}(x,y,t)])[\Phi(x,y,t)]\di x\di\sigma\di t
\\
-\alpha \int_{0}^{1}\int_{\Om}\int_{\Permemb}[u_{\delta }^{1,\#}(x,y,t)]
\frac{\partial }{\partial t}[\Phi (x,y,t)]\di x\di\sigma\di t=0\,,
\end{multline}
for every $(\phi,\Phi)\in\spaziogperoaa$, $[\Phi_t]\in L^2_\#\big(0,1;L^2(\Om\times\Permemb)\big)$
(recall Remark \ref{r:rem12}).
Moreover $u_{\delta }^{1,\#}$ has zero mean value in $Y$ for a.e. $(x,t)\in\Om\times\R$
and $u^\#_\delta$ satisfies \eqref{eq:BoundData_limit_perdelta} in the trace sense.
By \eqref{eq:a35} we get that $(u_{\delta }^{\#},u_\delta^{1,\#})$ satisfies an energy
estimate, easily obtained replacing $(\phi ,\Phi )$ with
$(u_{\delta }^{\#}-\dato,u_\delta^{1,\#})$, which implies
\begin{multline}\label{eq:a56}
\int_{0}^{1}\int_{\Om}\int_{Y}\frac{\dfboth}{2}|\nabla u_{\delta }^{\#}+\nabla_{y}u_\delta^{1,\#}|^{2}\di x\di y\di t
+\int_{0}^{1}\int_{\Om}\int_{\Permemb}f_{\delta}([u_\delta^{1,\#}])[u_\delta^{1,\#}]\di x\di\sigma\di t
\\
=\int_{0}^{1}\int_{\Om}\int_{Y}\frac{\dfboth}{2}|\nabla \dato|^{2}\di x\di y\di t\,,
\end{multline}
where we take into account
\begin{equation}\label{eq:star}
\int_0^1 [u_{\delta,t }^{1,\#}][u_{\delta }^{1,\#}]\di t =
\frac{1}{2}\int_0^1 \frac{\partial}{\partial t}[u_{\delta }^{1,\#}]^2
\di t =0\,,
\end{equation}
which is a consequence of the periodicity of $u_{\delta }^{1,\#}$.

From \eqref{eq:a56}, working as done in \eqref{eq:a39}--\eqref{eq:a40} of Section \ref{s:asymptotic1}
and taking into account \eqref{eq:a36} we get
\begin{equation}\label{eq:a57}
\int_{0}^{1}\int_{\Om}\int_{Y}\dfboth| \nabla u_{\delta}^{\#}+\nabla_y u_{\delta }^{1,\#}|^{2}\di x \di y\di t
+\int_{0}^{1}\int_{\Om}\int_{\Permemb}\lambda_1[u_{\delta }^{1,\#}]^{2}\di\sigma \di t
\leq \gamma\,,
\end{equation}
where $\gamma$ is a constant depending on $\lambda_1, \lambda_2,|\Permemb|$ and the $H^{1}$-norm
of $\dato $.

Replacing $(\phi ,\Phi )$ in \eqref{eq:a35} with $(u_{\delta, t}^{\#}-\dato_{t},u_{\delta, t}^{1,\#})$,
by \eqref{eq:a57}, \eqref{eq:a36} and taking into account the fact that
\begin{equation*}
\int_0^1 ( \nabla u_{\delta}^{\#}+\nabla_y u_{\delta }^{1,\#})(\nabla u_{\delta,t}^{\#}+\nabla_y u_{\delta,t}^{1,\#})
\di t=
\frac{1}{2}\int_0^1 \frac{\partial}{\partial t}| \nabla u_{\delta}^{\#}+\nabla_y u_{\delta }^{1,\#}|^{2}\di t =0
\end{equation*}
and, denoting by $F_\delta$ a primitive of $f_\delta$,
\begin{equation*}
\int_{0}^{1}\int_{\Om}\int_{\Permemb}f_\delta([u_{\delta }^{1,\#}])[u_{\delta,t}^{1,\#}]\di x\di\sigma\di t
=\int_{0}^{1}\int_{\Om}\int_{\Permemb} \frac{\partial F_\delta([u_{\delta }^{1,\#}])}{\partial t} \di x\di\sigma
\di t=0\,,
\end{equation*}
because of the periodicity, we get
\begin{equation}\label{eq:a58}
\alpha \int_{0}^{1}\int_{\Om}\int_{\Permemb}[u_{\delta,t}^{1,\#}]^{2}\di x\di\sigma\di t
\leq \gamma\,,
\end{equation}
where, again $\gamma$ depends on $\lambda_1,\lambda_2,|\Permemb|$ and the $H^{1}$-norms of
$\dato $ and $\dato_{t}$.
From \eqref{eq:a57}, we obtain
\begin{alignat}1
\label{eq:a60}
 &\int_{0}^{1}\int_{\Om}| \nabla u_{\delta }^{\#}|^{2}\di x\di t
\leq \gamma\,,
\\
\label{eq:a61}
& \int_{0}^{1}\int_{\Om}\int_{Y}| \nabla_y u_{\delta }^{1,\#}|^{2}\di x\di y\di t
\leq \gamma\,.
\end{alignat}
Indeed,
\begin{equation}\label{eq:a20}
\begin{aligned}
& \int_0^1\int_\Om\int_{Y}| \nabla_y u_{\delta }^{1,\#}(x,y,t)|^{2}\di y\di x\di t+
\int_0^1\int_{\Om}| \nabla u_{\delta }^{\#}|^{2}\di x\di t
\\
\leq &\gamma-2\int_0^1\int_{\Om}\int_Y\nabla_y u_{\delta }^{1,\#}(x,y,t)\nabla u_{\delta }^{\#}(x,t)\di y\di x\di t
\\
= & \gamma-
2\int_0^1\int_{\Om}\nabla u_{\delta }^{\#}\left( \int_{Y}\nabla_y u_{\delta }^{1,\#}(x,y,t)\di y\right) \di x\di t
\\
\leq &\gamma+
2\int_0^1\int_{\Om}|\nabla u_{\delta }^{\#}|\left( \int_{\Permemb}|[ u_{\delta }^{1,\#}(x,y,t)]|\di \sigma\right)\di x\di t
\\
\leq & \gamma+\frac{1}{2}\int_0^1\int_{\Om}| \nabla u_{\delta }^{\#}|^{2}\di x \di t+
8|\Permemb|\int_0^1\int_{\Om}\int_{\Permemb}[ u_{\delta }^{1,\#}(x,y,t)]^{2}\di\sigma \di x\di t
\\
\leq &\gamma+\frac{1}{2}\int_0^1\int_{\Om}| \nabla u_{\delta }^{\#}(x,t)|^{2}\di x\di t+\gamma\,.
\end{aligned}
\end{equation}
In order to be able to pass to the limit $\delta\to 0$ we need a formulation with vanishing boundary data. To this
purpose we set $v^\#_\delta= u^\#_\delta-\dato$; clearly $v^\#_\delta$ satisfies
 \begin{alignat}2
  \label{eq:PDE_lpd}
  -\Div\left(\dfav \nabla v_\delta^{\#}+\int_Y\dfboth\nabla_yu_\delta^{1,\#}\di y\right)&=
  \Div\left(\dfav\nabla \dato\right)\,,& \quad
  \text{in }\Om &\times\R\,;\\
  \label{eq:PDEper_lpd}
  -\Div_y(\dfboth\nabla v_\delta^{\#}+\dfboth \nabla_y u_\delta^{1,\#})&=0\,,&
  \!\!\!\!\!\!\!\!\!\!\!\!\!\!\!\!\!\!\!\text{in }\Om\times(\Perint\cup\Perout) &\times\R\,;\\
  \label{eq:FluxCont_lpd}
  [\dfboth (\nabla v_\delta^{\#}+\nabla_y u_\delta^{1,\#})\!\! \cdot\!\! \nu] &=
  -[\dfboth\nabla\dato\cdot\nu]\,,& \quad
  \text{on }\Om\times\Permemb &\times\R\,;\\
  \label{eq:Circuit_lpd}
   {\alpha}\frac{\partial}{\partial t}[u_\delta^{1,\#}] +
   f_\delta\left({[u_\delta^{1,\#}]}\right) &=\dfout (\nabla v_\delta^{\#}+\nabla_y u_\delta^{1,\#}) \cdot \nu
   +\dfout\nabla\dato\cdot\nu\,,&
  \text{on }\Om\times\Permemb &\times\R\,;\\
  \label{eq:InitData_lpd}
  [u_\delta^{1,\#}](x,y,\cdot)& \text{ is $1$-periodic,}&   \text{on }\Om & \times\Permemb\,;\\
  \label{eq:BoundData_lpd}
  v_\delta^{\#}(x,t)&=0\,,& \text{on }\partial\Om &\times\R\,,
\end{alignat}
or, in the weak form,
\begin{multline}\label{eq:a35v}
\int_{0}^{1}\int_{\Om}\int_{Y}\dfboth
\left( \nabla v_{\delta }^{\#}(x,t)+\nabla _{y}u_{\delta }^{1,\#}(x,y,t)\right)
\left( \nabla \phi (x,t)+\nabla_{y}\Phi (x,y,t)\right) \di x\di y\di t
\\
+\int_{0}^{1}\int_{\Om}\int_{\Permemb}
f_\delta([u_{\delta }^{1,\#}(x,y,t)])[\Phi(x,y,t)]\di x\di\sigma\di t
\\
-\alpha \int_{0}^{1}\int_{\Om}\int_{\Permemb}[u_{\delta }^{1,\#}(x,y,t)]
\frac{\partial }{\partial t}[\Phi (x,y,t)]\di x\di\sigma\di t
\\
= -\int_{0}^{1}\int_{\Om}\dfav \nabla\dato(x,t)\nabla \phi (x,t)\di x\di t
+\int_{0}^{1}\int_{\Om}\int_\Permemb [\dfboth\nabla\dato(x,t)\cdot \nu] \Phi^{(1)}(x,y,t)\di x\di \sigma\di t
\\
+\int_{0}^{1}\int_{\Om}\int_\Permemb \dfout\nabla\dato(x,t)\cdot \nu [\Phi(x,y,t)]\di x\di \sigma\di t
\,,
\end{multline}
for $(\phi,\Phi)$ as in Remark \ref{r:rem12}.
At this point, \eqref{eq:a57}--\eqref{eq:a61} allow us to pass to the limit with respect to
$\delta $ in the weak formulation \eqref{eq:a35v},
thus proving that there exists a periodic (in time) pair of functions $(v^{\#},u^{1,\#})
\in L^2_\#\big(0,1;H^1_0(\Om)\big)\times L^2_\#\big(0,1;L^2(\Om;\X^1_\#(Y))\big) $ such that
$u^{1,\#}$ has zero mean value on $Y$, for a.e. $(x,t)\in\Om\times\R$, and $(v^{\#},u^{1,\#})$ satisfies
the homogenized problem
\begin{multline}
\label{eq:a88}
  \int_{0}^{1}\int_{\Om\times Y} \dfboth (\nabla v^\# +\nabla_yu^{1,\#}) \cdot \nabla \phi \di
  x\di y\di t +
  \int_{0}^{1}\int_{\Om\times Y} \dfboth (\nabla v^\#+\nabla_yu^{1,\#})\cdot \nabla_y \Phi \di
  x\di y\di t
  \\
  +\int_{0}^{1}\int_\Om\int_{\Permemb}  \mu [\Phi]  \di x \di\sigma \di t
  - {\alpha}\int_{0}^{1}\int_\Om\int_{\Permemb} [u^{1,\#}] \pder{}{t}
  [\Phi] \di x\di\sigma\di t
  \\
  = -\int_{0}^{1}\int_{\Om}\dfav \nabla\dato\nabla \phi \di x\di t
+\int_{0}^{1}\int_{\Om}\int_\Permemb [\dfboth\nabla\dato\cdot \nu] \Phi^{(1)}\di x\di \sigma\di t
\\
+\int_{0}^{1}\int_{\Om}\int_\Permemb \dfout\nabla\dato\cdot \nu [\Phi]\di x\di \sigma\di t
\end{multline}
for every test function $(\phi,\Phi)\in\spaziogperoaa$ with $[\Phi_t]\in L^2_\#\big(0,1;L^2(\Om\times\Permemb)\big)$, where we have taken into account that
\eqref{eq:assumpt1} and \eqref{eq:a57} imply
\begin{equation*}
f_\delta ([u^{1,\#}_\delta]) =f ([u^{1,\#}_\delta]) +\delta [u^{1,\#}_\delta]\rightharpoonup  \mu\,,
\qquad\hbox{weakly in $L^2\big((0,T)\times\Om\times\Memb\big)$, when $\delta\to 0$.}
\end{equation*}
It remains to identify $\mu$. To this purpose, we follow the Minty monotone operators method.
Let us consider a sequence of $1$-periodic in time test functions
$\psi_k(x,y,t)=\phi^k_0(x,t)+\phi^k_1\left(x,y,t\right)+\lambda\phi_2\left(x,y,t\right)$,
with $\phi^k_0\in{\mathcal C}^\infty(\Om\times\R)$,
$\phi^k_1\in \CC^\infty\big(\Om\times\R;{\Cper}^\infty_\#(Y)\big)$,
$\phi_2\in {\mathcal C}^1_c\big(\Om\times(0,1);{\Cper}^1_\#(Y)\big)$,
with $\phi^k_0(\cdot, t)$ vanishing on $\partial\Om$ for $t\in\R$,
$\phi^k_0\to v^\#$ strongly in $L^2_{loc}\big(\R;H^1_0(\Om)\big)$, $\phi^k_1\to u^{1,\#}$ strongly in
$L^2_{loc}(\R;L^2(\Om;{\mathcal X}^1_\#(Y))\big)$, $[\phi^k_{1}]\to [u^{1,\#}]$ and $[\phi^k_{1,t}]\to [u^{1,\#}_t]$
strongly in $L^2((0,1)\times\Om\times\Permemb)$, i.e.
\begin{multline*}
\int_0^1 \int_\Om \Vert\phi^k_1(x,\cdot,t)- u^{1,\#}(x,\cdot,t)\Vert^2_{H^1(\Per_i)}\di t\di x
\\
+\int_0^1 \int_\Om \Vert [\phi^k_1(x,\cdot,t)]- [u^{1,\#}(x,\cdot,t)]\Vert^2_{L^2(\Permemb)}\di t\di x
\\
+\int_0^1 \int_\Om \Vert [\phi^k_{1,t}(x,\cdot,t)]- [u^{1,\#}_{t}(x,\cdot,t)]\Vert^2_{L^2(\Permemb)}\di t\di x
\to 0\,,\qquad\hbox{for $k\to +\infty$, $i=1,2$.}
\end{multline*}
Clearly, $\phi^k_0$ can be constructed by means of standard convolutions with regular kernels; instead,
in order to construct $\phi^k_1$ we proceed as follows.
Taking into account that,
passing to the limit for $\delta\to 0$ in \eqref{eq:a58}, we have
\begin{equation}\label{eq:a89}
[u^{1,\#}_{t}]\in L^2((0,1)\times\Om\times \Permemb)\,,
\end{equation}
by standard arguments we can approximate the jump $[u^{1,\#}]$ with a sequence
of $1$-periodic in time functions
$\widetilde\phi_1^k\in \CC^\infty(\Om\times\Permemb\times\R)$ such that
$\widetilde\phi_1^k\to [u^{1,\#}]$ strongly in $L^2\big((0,1)\times\Om;H^{1/2}(\Permemb)\big)$
and $\widetilde\phi_{1,t}^k\to [u^{1,\#}_t]$ strongly in $L^2\big((0,1)\times\Om\times \Permemb\big)$.
Now, define $\phi^k_1$ as the ($1$-periodic in time) solution of the problem
\begin{alignat}2
  -\Div_y\big(\dfboth (\nabla \phi^k_0+\nabla_y \phi^k_1)\big)&=0\,,&\qquad&\text{in $(\Perint\cup\Perout)\times\Om\times\R$;}
  \label{eq:perfi_pde}\\
  [\dfboth (\nabla \phi^k_0+\nabla_y \phi^k_1)\cdot\nu]&=-[\dfboth\nabla\dato\cdot\nu]\,,
  &\qquad&\text{on $\Permemb\times\Om\times\R$;}
  \label{eq:perfi_flux}\\
  [\phi^k_1]&=\widetilde\phi_1^k\,,&\qquad&\text{on $\Permemb\times\Om\times\R$;}
  \label{eq:perfi_jump}
\end{alignat}
and $\phi^k_1(x,\cdot,t)$ is $Y$-periodic with zero mean value on $Y$
for $(x,t)\in\Om\times\R$. By Lemma 7.3 in \cite{Amar:Andreucci:Bisegna:Gianni:2004a},
it follows that $\phi^k_1\in\CC^\infty\big(\Om\times[0,1];\Cper^\infty_\#(Y)\big)$.
Here, for the sake of simplicity, we work as if $\dato$ has enough regularity, otherwise
we proceed with a standard regularization procedure also on $\dato$.
Moreover, by \cite[Lemma 5]{Amar:Andreucci:Bisegna:Gianni:2005} applied to
$\phi^k_1-u^{1,\#}$ with $P=\Div_y\big(\dfboth (\nabla \phi^k_0-\nabla v^\#)\big)=0$ in $\Perint\cup\Perout$,
$Q=[\dfboth(\nabla\phi^k_0-\nabla v^\#)]$, and $S=\widetilde\phi_1^k-[u^{1,\#}]$, we obtain
\begin{equation}\label{eq:a89nuova}
||\phi^k_1-u^{1,\#}||_{L^2((0,1)\times\Om;\X^1_\#(Y))}\leq \gamma
(||\widetilde\phi^k_1-[u^{1,\#}]||_{L^2((0,1)\times\Om;H^{1/2}(\Permemb))}
+\Vert \nabla\phi^k_0-\nabla v^\#\Vert_{L^2((0,1)\times\Om)})\,.
\end{equation}
Since the right-hand side of \eqref{eq:a89nuova} tends to zero for $k\to +\infty$
we obtain the desired approximation.

Taking only into account the monotonicity assumption on $f$, the periodicity in time of
$\phi^k_{0}$ and $\phi^k_{1}$ and Remark \ref{r:rem12},  we calculate
\begin{multline}
\label{eq:monotonia1}
\int_{0}^1\int_{\Om\times Y}
\dfboth(\nabla v^\#_\delta +\nabla_y u^{1,\#}_{\delta}-\nabla\phi^k_0-\nabla_y\phi^k_1-\lambda\nabla_y\phi_2)
\cdot(\nabla v^\#_\delta-\nabla\phi^k_0)\di x\di y\di t
\\
+\int_{0}^1\int_{\Om\times Y}\dfboth(\nabla v^\#_\delta +\nabla_y u^{1,\#}_{\delta}-
\nabla\phi^k_0-\nabla_y\phi^k_1-\lambda\nabla_y\phi_2)
\cdot(\nabla_y u^{1,\#}_{\delta}-\nabla_y\phi^k_1-\lambda\nabla_y\phi_2)\di x\di y\di t
\\
 +{\alpha}\int_{0}^1\int_{\Om\times\Permemb} \frac{\partial}{\partial t}\left([u^{1,\#}_\delta]
 -[\phi^k_1+\lambda\phi_2]\right)\left([u^{1,\#}_\delta]-[\phi^k_1+\lambda\phi_2]\right)
 \di x\di\sigma\di t
\\
+ \int_{0}^1\int_{\Om\times\Permemb}\left(f_\delta([u^{1,\#}_\delta])
-{f_\delta\left([\phi^k_1+\lambda\phi_2]\right)}\right)
\left([u^{1,\#}_\delta]-[\phi^k_1+\lambda\phi_2]\right) \di x\di\sigma\di t
\\
=
\int_{0}^1\int_{\Om\times Y}
\dfboth|\nabla v^\#_\delta +\nabla_y u^{1,\#}_{\delta}-\nabla\phi^k_0-\nabla_y\phi^k_1-\lambda\nabla_y\phi_2|^2
\di x\di y\di t
\\
+ \int_{0}^1\int_{\Om\times\Permemb}\left(f_\delta([u^{1,\#}_\delta])
-{f_\delta\left([\phi^k_1+\lambda\phi_2]\right)}\right)
\left([u^{1,\#}_\delta]-[\phi^k_1+\lambda\phi_2]\right) \di x\di\sigma\di t
\geq 0\,,
\end{multline}
where we have taken into account that the time-periodicity of $u^{1,\#}_\delta$,
$\phi^k_1$ and $\phi_2$ implies
\begin{multline*}
{\alpha}\int_{0}^1\int_{\Om\times\Permemb} \frac{\partial}{\partial t}\left([u^{1,\#}_\delta]
-[\phi^k_1+\lambda\phi_2]\right)\left([u^{1,\#}_\delta]-[\phi^k_1+\lambda\phi_2]\right)
\di x\di\sigma\di t
\\
= \frac{\alpha}{2}\int_{0}^1\int_{\Om\times\Permemb} \frac{\partial}{\partial t}\left([u^{1,\#}_\delta]
-[\phi^k_1+\lambda\phi_2]\right)^2\di x\di\sigma\di t
\\
= \frac{\alpha}{2}\int_{\Om\times\Permemb} \left([u^{1,\#}_\delta(x,y,1)]
-[\phi^k_1(x,y,1)]\right)^2\di x\di\sigma
\\
-\frac{\alpha}{2}\int_{\Om\times\Permemb} \left([u^{1,\#}_\delta(x,y,0)]
-[\phi^k_1(x,y,0)]\right)^2\di x\di\sigma=0\,.
\end{multline*}
Taking the function $(v^{\#}_\delta-\phi_0^k,u^{1,\#}_\delta-\phi^k_1-\lambda\phi_2)$
as a test function $(\varphi,\Phi)$ in \eqref{eq:a35v},
inequality \eqref{eq:monotonia1} can be rewritten as
\begin{multline}
\label{eq:monotonia1bis}
-\int_{0}^1\int_{\Om\times Y}
\dfboth(\nabla\phi^k_0+\nabla_y\phi^k_1+\lambda\nabla_y\phi_2 )
\cdot(\nabla v^\#_\delta-\nabla\phi^k_0)\di x\di y\di t
\\
-\int_{0}^1\int_{\Om\times Y}\dfboth(\nabla\phi^k_0+\nabla_y\phi^k_1+\lambda\nabla_y\phi_2)
\cdot(\nabla_y u^{1,\#}_{\delta}-\nabla_y\phi^k_1-\lambda\nabla_y\phi_2)\di x\di y\di t
\\
-{\alpha}\int_{0}^1\int_\Om\int_{\Om\times\Permemb} \frac{\partial}{\partial t}
[\phi^k_1+\lambda\phi_2]\left([u^{1,\#}_\delta]-[\phi^k_1+\lambda\phi_2]\right)
\di x\di\sigma\di t
\\
- \int_{0}^1\int_{\Om\times\Permemb}{f_\delta\left([\phi^k_1+\lambda\phi_2]\right)}
\left([u^{1,\#}_\delta]-[\phi^k_1+\lambda\phi_2]\right)\di x \di\sigma\di t
\\
\geq
\int_{0}^{1}\int_{\Om}\dfav \nabla\dato\cdot(\nabla v^\#_\delta-\nabla\phi^k_0)\di x\di t
-\int_{0}^{1}\int_{\Om}\int_\Permemb [\dfboth\nabla\dato\cdot \nu] (u^{1,\#}_\delta-\phi^k_1-\lambda\phi_2)^{(1)}
\di x\di \sigma\di t
\\
-\int_{0}^{1}\int_{\Om}\int_\Permemb \dfout\nabla\dato\cdot \nu ([u^{1,\#}_\delta]-[\phi^k_1+\lambda\phi_2])
\di x\di \sigma\di t
\,.
\end{multline}
Hence, passing to the limit as $\delta\to 0$ and using \eqref{eq:a58}, it follows
\begin{multline}
\label{eq:monotonia2}
-\int_{0}^1\int_{\Om\times Y}
\dfboth(\nabla\phi^k_0+\nabla_y\phi^k_1+\lambda\nabla_y\phi_2)
\cdot(\nabla v^\#-\nabla\phi^k_0)\di x\di y\di t
\\
-\int_{0}^1\int_{\Om\times Y}\dfboth(\nabla\phi^k_0+\nabla_y\phi^k_1+\lambda\nabla_y\phi_2)
\cdot(\nabla_y u^{1,\#}-\nabla_y\phi^k_1-\lambda\nabla_y\phi_2)\di x\di y\di t
\\
-{\alpha}\int_{0}^1\int_\Om\int_{\Om\times\Permemb} \frac{\partial}{\partial t}
[\phi^k_1+\lambda\phi_2]\left([u^{1,\#}]-[\phi^k_1+\lambda\phi_2]\right)
\di x\di\sigma\di t
\\
- \int_{0}^1\int_{\Om\times\Permemb}{f\left([\phi^k_1+\lambda\phi_2]\right)}
\left([u^{1,\#}]-[\phi^k_1+\lambda\phi_2]\right)\di x \di\sigma\di t
\\
\geq \int_{0}^{1}\int_{\Om}\dfav \nabla\dato\cdot(\nabla v^\#-\nabla\phi^k_0)\di x\di t
-\int_{0}^{1}\int_{\Om}\int_\Permemb [\dfboth\nabla\dato\cdot \nu]
(u^{1,\#}-\phi^k_1-\lambda\phi_2)^{(1)}\di x\di \sigma\di t
\\
-\int_{0}^{1}\int_{\Om}\int_\Permemb \dfout\nabla\dato\cdot \nu ([u^{1,\#}]-[\phi^k_1+\lambda\phi_2])
\di x\di \sigma\di t\,.
\end{multline}
Now, letting $k\to +\infty$, we obtain
\begin{multline}
\label{eq:monotonia3}
\int_0^1\int_\Om\int_Y\dfboth(\nabla v^\#+\nabla_yu^{1,\#}+\lambda\nabla_y\phi_2)
\cdot\lambda\nabla_y\phi_2\di x\di y\di t +
\\
{\alpha}\int_0^1\int_\Om\int_{\Permemb}\frac{\partial}{\partial t}[u^{1,\#}+\lambda\phi_2]
\lambda[\phi_2]\di x\di\sigma\di t
+\int_0^1\int_\Om\int_{\Permemb}
{f\left({[u^{1,\#}+\lambda\phi_2]}\right)}
\lambda[\phi_2]\di x\di\sigma\di t
\\
\geq
\int_{0}^{1}\int_{\Om}\int_\Permemb [\dfboth\nabla\dato\cdot \nu] \lambda\phi_2^{(1)}
\di x\di \sigma\di t
+\int_{0}^{1}\int_{\Om}\int_\Permemb \dfout\nabla\dato\cdot \nu [\lambda\phi_2]\di x\di \sigma\di t\,.
\end{multline}
Taking into account \eqref{eq:a88} with $\phi\equiv0$ and $\Phi=\phi_2$, \eqref{eq:monotonia3} becomes
\begin{multline}
\label{eq:monotonia4}
\lambda^2\int_0^1\int_\Om\int_Y\dfboth\nabla_y\phi_2
\cdot\nabla_y\phi_2\di x\di y\di t
+{\alpha}\lambda^2\int_0^1\int_\Om\int_{\Permemb}\frac{\partial}{\partial t}[\phi_2]
[\phi_2]\di x\di\sigma\di t
\\
-\lambda\int_0^1\int_\Om\int_{\Permemb}
\mu[\phi_2]\di x\di\sigma\di t
+\lambda\int_0^1\int_\Om\int_{\Permemb}
{f\left({[u^{1,\#}+\lambda\phi_2]}\right)}
[\phi_2]\di x\di\sigma\di t\geq 0\,.
\end{multline}
Assuming firstly that $\lambda>0$ and then $\lambda<0$, dividing by $\lambda$ the previous equation and
then letting $\lambda\to 0$, we obtain
\begin{equation*}
\int_0^1\int_\Om\int_{\Permemb}\mu[\phi_2]\di x\di\sigma\di t
=\int_0^1\int_\Om\int_{\Permemb}{f\left({[u^{1,\#}]}\right)}[\phi_2]\di x\di\sigma\di t\,,
\end{equation*}
which gives
\begin{equation}
\label{eq:identificazione}
\mu=f\left({[u^{1,\#}]}\right)\,.
\end{equation}
By \eqref{eq:a88} and \eqref{eq:identificazione}, setting $v^\#=u^\#+\dato$ and taking into account Remark \ref{r:rem12},
we obtain exactly the weak formulation of problem
\eqref{eq:PDE_limit_per}--\eqref{eq:BoundData_limit_per}.
\end{proof}

\begin{remark}\label{r:rem10}
Note that \eqref{eq:a58} is uniform with respect to $\delta$. Moreover, we can obtain also
estimates for $\nabla u^\#_{\delta,t}$ and $\nabla_y u^{1,\#}_{\delta,t}$
uniformly in $\delta$. Indeed, differentiating formally with respect to $t$ problem
\eqref{eq:PDE_limit_perdelta}--\eqref{eq:BoundData_limit_perdelta}, multiplying equation \eqref{eq:PDE_limit_perdelta}
(differentiated with respect to $t$) by $\big((u^{\#}_{\delta,t}-\dato_{t}),
u^{1,\#}_{\delta,t}\big)$, and finally integrating by parts, we obtain, exploiting also the periodicity in time,
\begin{equation}\label{eq:a81}
\int_{0}^{1}\int_{\Om}\int_{Y}\dfboth| \nabla u^{\#}_{\delta,t}+\nabla_{y}u^{1,\#}_{\delta,t}|^{2}\di x\di y\di t
\leq {\gamma }\,
\end{equation}
where we used assumptions \eqref{eq:assumpt1}, \eqref{eq:h1} and inequality \eqref{eq:a58}. Now, proceeding as in the proof of \eqref{eq:a60} and
\eqref{eq:a61}, we obtain
\begin{alignat}1
\label{eq:a60bis}
 &\int_{0}^{1}\int_{\Om}| \nabla u_{\delta,t }^{\#}|^{2}\di x\di t
\leq \gamma\,,
\\
\label{eq:a61bis}
& \int_{0}^{1}\int_{\Om}\int_{Y}| \nabla_y u_{\delta,t }^{1,\#}|^{2}\di x\di y\di t
\leq \gamma\,.
\end{alignat}
Therefore, passing to the limit for $\delta\to 0^+$, in \eqref{eq:a58}, \eqref{eq:a60bis} and \eqref{eq:a61bis}, we obtain that
the same estimates hold for $(u^\#,u^{1,\#})$.

This implies that $(u^\#,u^{1,\#})$ belongs to $\spaziogperaa$.
\end{remark}

It remains to prove that any solution $(u,u^{1})$ of the homogenized
problem converges to $(u^{\#},u^{1,\#})$ as $t\rightarrow \infty $.
This is the purpose of the next theorem.

\begin{thm}\label{t:t6}
Let $(u,u^1)\in L^2\big(0,T);H^1(\Om)\big)\times L^2\big(\Om\times(0,T);\X^1_\#(Y)\big)$ be the solution of problem
\eqref{eq:PDE_limit}--\eqref{eq:BoundData_limit}. Then, for $t\to +\infty$, $(u,u^1)\to
(u^{\#},u^{1,\#})$ in the following sense:
\begin{alignat}1
\label{eq:decayperiodic_neww}
& \lim_{t\to+\infty}\| u(\cdot,t)-u^{\#}(\cdot,t)\| _{H^{1}(\Om)}=0\,;
\\
\label{eq:decayperiodic1_neww}
& \lim_{t\to+\infty}\left[\| u^1(\cdot,\cdot,t)-u^{1,\#}(\cdot,\cdot,t)\| _{L^{2}(\Om\times Y)}
+\| \nabla_y u^1(\cdot,\cdot,t)-\nabla_y u^{1,\#}(\cdot,\cdot,t)\| _{L^{2}(\Om\times Y)}\right]=0\,;
\\
\label{eq:decayperiodic3_neww}
& \lim_{t\to+\infty}\|[u^1](\cdot,\cdot,t)-[u^{1,\#}](\cdot,\cdot,t)\|_{L^{2}(\Om\times\Permemb)}=0\,.
\end{alignat}
\end{thm}

\begin{proof}
Firstly we recall that, by \cite[Lemma 5.2]{Amar:Andreucci:Gianni:2014a}
which holds even in the present case,
$(u,u^1)\in \CC^0\big((0,T];H^1(\Om)\big)\times \CC^0\big((0,T];L^2(\Om;\X^1_\#(Y))\big)$ and
$[u^1]\in \CC^0\big((0,T];L^2(\Om\times \Permemb)\big)$.

As usual, let $(r,r^{1}):=(u^{\#}-u,u^{1,\#}-u^{1})$, so that the pair $(r,r^{1})$ satisfies:
\begin{multline}\label{eq:a62}
\int_{0}^{t}\int_{\Om}\int_{Y}\dfboth\left(\nabla r+\nabla_{y}r^{1}\right) \left(\nabla\phi+\nabla_{y}\Phi \right)
\di x\di y\di t
\\
+\int_{0}^{t}\int_{\Om}\int_{\Permemb}\frac{f([u^{1,\#}])-f([u^{1}])}{[u^{1,\#}]-[u^{1}]}[r^{1}][\Phi]
\di x\di \sigma\di t
+\alpha \int_{0}^{t}\int_{\Om}\int_{\Permemb}[r^{1}_t][\Phi]\di x\di \sigma\di t=0\,,\qquad\forall t\in(0,T)\,,
\end{multline}
where $r=0$ on $\partial \Omega \times [0,T]$ in the trace sense,
$r^{1}$ is periodic in $Y$ and has zero mean value in $Y$ for almost every $(x,t)\in\Om\times (0,T)$.
Here $\phi $ is any regular function depending on $(x,t)$, with compact
support in $\Om$ and $\Phi $ is a any function depending on $(x,y,t)$
which jumps across $\Permemb$, is zero when $t=T$ and is regular
elsewhere.
Differentiating \eqref{eq:a62} with respect to $t$, we get
\begin{multline}\label{eq:a63}
\int_{\Om}\int_{Y} \dfboth \left( \nabla r+\nabla_{y}r^{1}\right) \left(\nabla \phi +\nabla_{y}\Phi \right) \di x\di y
+\int_{\Om}\int_{\Permemb}\frac{f([u^{1,\#}])-f([u^{1}])}{[u^{1,\#}]-[u^{1}]}[r^{1}][\Phi ]\di x\di\sigma
\\
+\alpha \int_{\Om}\int_{\Permemb}[r^{1}_t][\Phi ]\di x\di \sigma =0\,.
\end{multline}
Replacing $\left( \phi ,\Phi \right) $ with $\left(r,r^{1}\right)$ in \eqref{eq:a63}, we get
\begin{multline}\label{eq:a64}
\int_{\Om}\int_{Y}\dfboth |\nabla r+\nabla_{y}r^{1}|^{2}\di x\di y
+\int_{\Om}\int_{\Permemb}\frac{f([u^{1,\#}])-f([u^{1}])}{[u^{1,\#}]-[u^{1}]}[r^{1}]^{2}\di x\di\sigma
\\
+\alpha \int_{\Om}\int_{\Permemb}[r^{1}_t][r^{1}]\di x\di\sigma =0\,.
\end{multline}
As in Section \ref{s:asymptotic1}, equation \eqref{eq:a64} implies that
the function $t\mapsto \alpha \int_{\Om}\int_{\Permemb}[ r^{1}(x,t)]^{2}\di\sigma \di x$ is a positive,
decreasing function of $t$, hence it tends to a limit value $\overline r^1\geq 0$ as $t\to +\infty $.
The value $\overline r^1$ must be zero otherwise
$\alpha\int_{\Om}\int_{\Permemb}[r^{1}]^{2}\di\sigma \di x\geq {\overline r^1}>0$
for every $t>0$. On the other hand, for $t>0$ and setting $\Permemb_{\overline r^1}(t):=\left\{ (x,y)\in \Om\times \Permemb
:[r^1]^{2}(x,y,t)\leq \frac{\overline r^1}{2\alpha |\Permemb|\,|\Om|}\right\}$,
reasoning as in the proof of Theorem \ref{t:t5}, it follows that
\begin{equation*}
\alpha \int_{\Om}\int_{\Permemb\setminus\Permemb_{\overline r^1}(t)}[r^{1}(x,y,t)]^{2}\di\sigma \di x
\geq \frac{\overline r^1}{2}\,, \quad \forall t>0\,.
\end{equation*}
However, on $\Permemb\setminus\Permemb_{\overline r^1}$, $g(x,y,t):=\frac{f([u^{1,\#}])-f([u^{1}])}{[u^{1,\#}]-[u^{1}]}\geq \chi>0$, where $\chi $
is a suitable positive constant depending only on $\overline r^1,\alpha ,|\Permemb|, |\Om|$
(this last result follows from the assumptions \eqref{eq:a37}--\eqref{eq:a38}).
Hence, using \eqref{eq:a64}, we get
\begin{multline}\label{eq:a65}
\frac{d}{dt}\left( \frac{\alpha}{2} \int_{\Om}\int_{\Permemb}[r^{1}(x,y,t)]^{2}\di\sigma \di x\right)
\leq -\int_{\Om}\int_{\Permemb\setminus\Permemb_{\overline r^1}(t)}g(x,y,t)
[r^{1}(x,y,t)]^{2}\di \sigma \di x
\\
\leq -\chi\int_{\Om}\int_{\Permemb\setminus\Permemb_{\overline r^1}(t)}
[r^{1}(x,y,t)]^{2}\di\sigma \di x
\leq -\frac{\overline r^1}{2\alpha }\chi<0\,.
\end{multline}
Inequality \eqref{eq:a65} clearly contradicts the asymptotic convergence for $t\to +\infty$
of $\alpha \int_{\Om}\int_{\Permemb}[ r^{1}]^{2}(x,y,t)\di\sigma \di x$ to a positive number, hence
\begin{equation}\label{eq:a66}
\lim_{t\to +\infty }\alpha \int_{\Om}\int_{\Permemb}[r^{1}(x,y,t)]^{2}\di\sigma\di x=0\,,
\end{equation}
which is exactly \eqref{eq:decayperiodic3_neww}.
Integrating \eqref{eq:a64} in $[t,\infty )$ and taking into account \eqref{eq:a66},
we get
\begin{equation}\label{eq:a67}
\int_{t}^{+\infty }\int_{\Om}\int_{Y}\dfboth |\nabla r+\nabla_{y}r^{1}|^{2}\di x\di y\di t
\leq \frac{\alpha }{2}\int_{\Om}\int_{\Permemb}[r^{1}(x,y,t)]^{2}\di\sigma \di x\,,
\end{equation}
which implies
\begin{equation}\label{eq:a68}
\lim_{t\to+\infty }\int_{t}^{+\infty }\int_{Y}\int_{\Om}\dfboth|\nabla r+\nabla_{y} r^{1}|^{2}\di x\di y\di t=0\,.
\end{equation}
This last condition guarantees that for every positive $\eta$ there exists a $\widehat{t}(\eta)>0$, such that
\begin{equation*}
\int_{\widehat{t}}^{+\infty}\int_{\Om}\int_{Y}\dfboth|\nabla r+\nabla_{y}r^{1}|^{2}\di x\di y\di t
\leq \eta\,,
\end{equation*}
which in turn implies that, for every $n\in\NN$, there exists a $t_{n}\in (\widehat{t}+n,\widehat{t}+(n+1))$,
such that
\begin{equation}\label{eq:a69}
\int_{\Om}\int_{Y}\dfboth|\nabla r(x,t_n)+\nabla_{y}r^{1}(x,y,t_n)|^{2}\di x\di y
\leq \eta\,.
\end{equation}
Hence, replacing $( \phi ,\Phi) $ with $( r_{t},r_{t}^{1}) $ in \eqref{eq:a63}, we get
\begin{multline}\label{eq:a70}
\int_{\Om}\int_{Y}\dfboth (\nabla r+\nabla_{y}r^{1})(\nabla r_{t}+\nabla_{y}r_{t}^{1})\di x\di y
+\int_{\Om}\int_{\Permemb} g(x,y,t)[r^{1}][r_{t}^{1}] \di \sigma \di x
\\
+\alpha \int_{\Om}\int_{\Permemb}[ r_{t}^{1}(x,y,t)]^{2}\di\sigma \di x=0\,,
\end{multline}
and
\begin{equation}\label{eq:a71}
\int_{\Om}\int_{Y}\dfboth (\nabla r+\nabla_{y}r^{1})(\nabla r_{t}+\nabla_{y}r_t^{1})\di x\di y
\leq \int_{\Om}\int_{\Permemb}\frac{g^{2}(x,y,t)}{2\alpha }[ r^{1}(x,y,t)]^{2}\di \sigma \di x\,.
\end{equation}
Moreover, integrating \eqref{eq:a71} in $[t_n,t^*]$, with $t^*\in[t_{n},t_{n}+2]$, we have
\begin{multline}\label{eq:a72}
\sup_{t\in [t_{n},t_{n}+2]}\left(\int_{\Om}\int_{Y}\frac{\dfboth}{2}
| \nabla r(x,t)+\nabla_{y}r^{1}(x,y,t)|^{2}\di x\di y\right)
\\
\leq \frac{\eta}{2}+\frac{2L^{2}}{2\alpha^{2}}\sup_{t\in[t_{n},+\infty )}
\left( \alpha \int_{\Om}\int_{\Permemb}[r^{1}(x,y,t)]^{2}\di\sigma \di x\right)\,,
\qquad \forall n\in\NN\,;
\end{multline}
i.e.,
\begin{multline}\label{eq:a73}
\sup_{t\in [\widehat{t}+1,+\infty )}\left( \int_{\Om}\int_{Y}\frac{\dfboth}{2}
|\nabla r(x,t)+\nabla_{y}r^{1}(x,y,t)|^{2}\di x\di y\right)
\\
\leq \frac{\eta}{2}+\frac{L^{2}}{\alpha^{2}}\sup_{t\in [\widehat{t},+\infty )}
\left( \alpha \int_{\Om}\int_{\Permemb}[r^{1}(x,y,t)]^{2}\di\sigma\right).
\end{multline}
Because of \eqref{eq:a66} the integral in the right-hand side of \eqref{eq:a73} can be made
smaller than $\frac{\eta}{2}\left( \frac{L^{2}}{\alpha^{2}}\right)^{-1}$,
provided $\widehat{t}$ is chosen sufficiently large in dependence of $\eta$.
This means that
\begin{equation}\label{eq:a74}
\sup_{t\in [\widehat{t}+1,+\infty )}\left(\int_{\Om} \int_{Y}\frac{\dfboth}{2}
| \nabla r(x,t)+\nabla_{y}r^{1}(x,y,t)|^{2}\di x\di y\right) \leq \eta\,.
\end{equation}
\bigskip
Inequality \eqref{eq:a74} implies
\begin{equation}\label{eq:a75}
\lim_{t\to +\infty }\int_{\Om} \int_{Y}\dfboth|\nabla r(x,t)+\nabla_{y}r^{1}(x,y,t)|^{2}\di x\di y=0\,.
\end{equation}
Now, working as done in \eqref{eq:a20},
we get
\begin{equation*}
\lim_{t\to +\infty }\int_{\Om}|\nabla r(x,t)|^{2}\di x\di y =0\,;
\qquad\text{and}\qquad
\lim_{t\to +\infty }\int_{\Om}\int_{Y}|\nabla_{y}r^{1}(x,y,t)|^{2}\di x\di y=0\,.
\end{equation*}
Finally, the previous results together with \eqref{eq:a66}  and Poincare's inequalities yield
\begin{equation*}
\lim_{t\to+\infty }\int_{\Om}|r(x,t)|^{2}\di x=0\,;
\qquad\text{and}\qquad
\lim_{t\to+\infty }\int_{\Om}\int_{Y}|r^{1}(x,y,t)|^{2}\di x=0\,,
\end{equation*}
which give \eqref{eq:decayperiodic_neww} and \eqref{eq:decayperiodic1_neww} and conclude the proof.
\end{proof}


\begin{remark}\label{r:rem10trisnew}
More in general, the previous procedure allows us to prove that solutions of \eqref{eq:PDE_limit}--\eqref{eq:BoundData_limit}
having different initial data satisfying the assumptions stated at the beginning of this section
but with the same boundary condition tend asymptotically one to the other (such convergence  being exponential if $f$ is coercive in the sense of \eqref{eq:intro1}).
\end{remark}

\begin{remark}\label{r:rem10bisnew}
Observe that, thanks to previous remark, Theorem \ref{t:t5} implies uniqueness of the periodic solution $(u^{\#},u^{1,\#})$
of problem \eqref{eq:PDE_limit_per}--\eqref{eq:BoundData_limit_per} in $\CC_\#^0\big([0,1];H^1(\Om)\big)
\times \CC^0_\#\big([0,1];L^2(\Om;\X^1_\#(Y))\big)$.
\end{remark}

\bibliographystyle{abbrv}
\bibliography{damiro1}

\vfill\eject
\end{document}

%% file: fig_omega.tex
\begin{pspicture}(12,6)
\rput(0,1){
\psframe[linewidth=0.4pt,linestyle=dashed](0,0)(4,4)
\psccurve[fillstyle=solid,fillcolor=lightgray](2,.5)(2.5,2)(2,3.5)(1.5,3)(1.5,2)(1,1)
}
\newcommand{\minicell}{\scalebox{0.1}{
\psccurve[fillstyle=solid,fillcolor=lightgray](2,.5)(2.5,2)(2,3.5)(1.5,3)(1.5,2)(1,1)
}}
\rput(6,1){
\multirput(0,0.4)(0.4,0){3}{\minicell}%
\multirput(2.4,0.4)(0.4,0){7}{\minicell}%
\multirput(0.4,0.8)(0.4,0){12}{\minicell}%
\multirput(0,1.2)(0,0.4){5}{\multirput(0,0)(0.4,0){14}{\minicell}}%
\multirput(0.4,3.2)(0.4,0){13}{\minicell}%
\multirput(0.4,3.2)(0.4,0){13}{\minicell}%
\multirput(0.4,3.6)(0.4,0){13}{\minicell}%
\multirput(0.8,4)(0.4,0){12}{\minicell}%
\multirput(1.6,4.4)(0.4,0){9}{\minicell}%
\multirput(2.4,4.8)(0.4,0){6}{\minicell}%
\psccurve(0.2,.1)(1.5,.5)(5,.3)(5,5)(.5,4)(0,1)
}

\end{pspicture}

%% file: damiro1_secondo.bbl
\begin{thebibliography}{10}

\bibitem{Allaire:1992}
G.~Allaire.
\newblock Homogenization and two-scale convergence.
\newblock {\em SIAM J. Math. Anal.}, 23:1482--1518, 1992.

\bibitem{Allaire:Briane:1996}
G.~Allaire and M.~Briane.
\newblock Multi-scale convergence and reiterated homogenization.
\newblock {\em Proc. Roy. Soc. Edinburgh}, 126A:297--342, 1996.

\bibitem{Allaire:Damlamian:Hornung:1995}
G.~Allaire, A.~Damlamian, and U.~Hornung.
\newblock Two-scale convergence on periodic surfaces and applications.
\newblock {\em Proceedings of the International Conference on Mathematical
  Modelling of Flow through Porous Media}, 15--25, 1995.

\bibitem{Amar:Andreucci:Bisegna:Gianni:2003b}
M.~Amar, D.~Andreucci, P.~Bisegna, and R.~Gianni.
\newblock Evolution and memory effects in the homogenization limit for
  electrical conduction in biological tissues: the $1$-d case.
\newblock {\em Proceedings del XVI Congresso AIMETA di Meccanica Teorica e
  Applicata, Ferrara}, 2003.

\bibitem{Amar:Andreucci:Bisegna:Gianni:2003a}
M.~Amar, D.~Andreucci, P.~Bisegna, and R.~Gianni.
\newblock Homogenization limit for electrical conduction in biological tissues
  in the radio-frequency range.
\newblock {\em Comptes Rendus Mecanique}, 331:503--508, 2003.
\newblock Elsevier.

\bibitem{Amar:Andreucci:Bisegna:Gianni:2004b}
M.~Amar, D.~Andreucci, P.~Bisegna, and R.~Gianni.
\newblock An elliptic equation with history.
\newblock {\em C.\ R.\ Acad.\ Sci.\ Paris, Ser.\ I}, 338:595--598, 2004.
\newblock Elsevier.

\bibitem{Amar:Andreucci:Bisegna:Gianni:2004a}
M.~Amar, D.~Andreucci, P.~Bisegna, and R.~Gianni.
\newblock Evolution and memory effects in the homogenization limit for
  electrical conduction in biological tissues.
\newblock {\em Mathematical Models and Methods in Applied Sciences},
  14:1261--1295, 2004.
\newblock World Scientific.

\bibitem{Amar:Andreucci:Bisegna:Gianni:2005}
M.~Amar, D.~Andreucci, P.~Bisegna, and R.~Gianni.
\newblock Existence and uniqueness for an elliptic problem with evolution
  arising in electrodynamics.
\newblock {\em Nonlinear Analysis Real World Applications}, 6:367--380, 2005.
\newblock Elsevier.

\bibitem{Amar:Andreucci:Bisegna:Gianni:2006a}
M.~Amar, D.~Andreucci, P.~Bisegna, and R.~Gianni.
\newblock On a hierarchy of models for electrical conduction in biological
  tissues.
\newblock {\em Mathematical Methods in the Applied Sciences}, 29:767--787,
  2006.

\bibitem{Amar:Andreucci:Bisegna:Gianni:2009}
M.~Amar, D.~Andreucci, P.~Bisegna, and R.~Gianni.
\newblock Exponential asymptotic stability for an elliptic equation with memory
  arising in electrical conduction in biological tissues.
\newblock {\em Euro. Jnl. of Applied Mathematics}, 20:431--459, 2009.

\bibitem{Amar:Andreucci:Bisegna:Gianni:2009a}
M.~Amar, D.~Andreucci, P.~Bisegna, and R.~Gianni.
\newblock Stability and memory effects in a homogenized model governing the
  electrical conduction in biological tissues.
\newblock {\em J. Mechanics of Material and Structures}, (2) 4:211--223, 2009.

\bibitem{Amar:Andreucci:Bisegna:Gianni:2010}
M.~Amar, D.~Andreucci, P.~Bisegna, and R.~Gianni.
\newblock Homogenization limit and asymptotic decay for electrical conduction
  in biological tissues in the high radiofrequency range.
\newblock {\em Communications on Pure and Applied Analysis}, (5) 9:1131--1160,
  2010.

\bibitem{Amar:Andreucci:Bisegna:Gianni:2013}
M.~Amar, D.~Andreucci, P.~Bisegna, and R.~Gianni.
\newblock A hierarchy of models for the electrical conduction in biological
  tissues via two-scale convergence: The nonlinear case.
\newblock {\em Differential and Integral Equations}, (9-10) 26:885--912, 2013.

\bibitem{Amar:Andreucci:Gianni:2014a}
M.~Amar, D.~Andreucci, and R.~Gianni.
\newblock Exponential decay for a nonlinear model for electrical conduction in
  biological tissues.
\newblock {\em To appear}, 2015.

\bibitem{Auriault:Ene:1994}
J.~Auriault and H.~Ene.
\newblock Macroscopic modelling of heat transfer in composites with interfacial
  thermal barrier.
\newblock {\em Int. J. Heat Mass Transfer}, 37(18):2885--2892, 1994.

\bibitem{Bellieud:Bouchitte:1998}
M.~Bellieud and G.~Bouchitt\'{e}.
\newblock Homogenization of elliptic problems in a fiber reinforced structure.
  non local effects.
\newblock {\em Ann. Scuola Norm. Sup. Pisa Cl. Sci}, XXVI(4):407--436, 1998.

\bibitem{Bisegna:Caruso:Lebon:2000}
P.~Bisegna, G.~Caruso, and F.~Lebon.
\newblock Bioelectrical impedance analysis: a matter of homogenization of
  composites with imperfect interfaces.
\newblock In G.~Augusti, editor, {\em Proceedings 15th AIMETA Congress of
  Theoretical and Applied Mechanics}. 2001.

\bibitem{Bellieud:Bouchitte:2002}
G.~Bouchitt\'{e} and M.~Bellieud.
\newblock Homogenization of a soft elastic material reinforced by fibers.
\newblock {\em Asymptotic Analysis}, 32:153--183, 2002.

\bibitem{Bronzino:1999}
J.~D. Bronzino.
\newblock {\em The Biomedical Engineering Handbook}.
\newblock CRC Press, 1999.

\bibitem{Cioranescu:Donato:1988}
D.~Cioranescu and P.~Donato.
\newblock Homog\'{e}n\'{e}isation du probl\`{e}me de {N}eumann non homog\`{e}ne
  dans des ouverts perfor\'{e}s.
\newblock {\em Asymptotic Anal.}, 1:115--138, 1988.

\bibitem{Clark:Packer:1997}
G.~Clark and L.~Packer.
\newblock Two-scale homogenization of implicit degenerate evolution equations.
\newblock {\em Journal of Mathematical Analysis and Applications},
  214:420--438, 1997.

\bibitem{DeLorenzo:1997}
A.~De~Lorenzo, A.~Andreoli, J.~Matthie, and P.~Withers.
\newblock Predicting body cell mass with bioimpedence by using theoretical
  methods: a technological review.
\newblock {\em J. Appl. Physiol.}, 82:1542--1558, 1997.

\bibitem{Fabrizio:Lazzari:1986}
M.~Fabrizio and B.~Lazzari.
\newblock Sulla stabilit\`a di un sistema viscoelastico lineare.
\newblock In {\em Acc. Naz. Lincei, Tavola rotonda sul tema: Continui con
  Memoria}. Roma, 1992.

\bibitem{Fabrizio:Morro:1988}
M.~Fabrizio and A.~Morro.
\newblock Viscoelastic relaxation functions compatible with thermodynamics.
\newblock {\em Journal of Elasticity}, 19:63--75, 1988.

\bibitem{Foster:Schwan:1989}
K.~R. Foster and H.~P. Schwan.
\newblock Dielectric properties of tissues and biological materials: a critical
  review.
\newblock {\em Critical Reviews in Biomedical Engineering}, 17:25--104, 1989.

\bibitem{Giorgi:Naso:Pata:2001}
C.~Giorgi, M.~G. Naso, and V.~Pata.
\newblock Exponential stability in linear heat conduction with memory: a
  semigroup approach.
\newblock {\em Commun. Appl. Anal.}, 5(1):121--133, 2001.

\bibitem{Giorgi:Naso:Pata:2005}
C.~Giorgi, M.~G. Naso, and V.~Pata.
\newblock Energy decay of electromagnetic systems with memory.
\newblock {\em Math. Models Methods Appl. Sci.}, 15(10):1489--1502, 2005.

\bibitem{Hummel:2000}
H.-K. Hummel.
\newblock Homogenization for heat transfer in polycrystals with interfacial
  resistances.
\newblock {\em Appl. Anal.}, 75:403--424, 2000.

\bibitem{Kandell:2000}
E.~R. Kandell, J.~H. Schwartz, and T.~M. Jessell, editors.
\newblock {\em Principles of Neural Science}.
\newblock McGraw-Hill, New York, 2000.
\newblock Fourth edition.

\bibitem{Krassowska:Neu:1993}
W.~Krassowska and J.~C. Neu.
\newblock Homogenization of syncytial tissues.
\newblock {\em Critical Reviews in Biomedical Engineering}, 21:137--199, 1993.

\bibitem{Lene:Leguillon:1981}
F.~Lene and D.~Leguillon.
\newblock \'{E}tude de l'influence d'un glissement entre les constituants d'un
  mat\'{e}riau composite sur ses coefficients de comportement effectifs.
\newblock {\em Journal de M\'{e}canique}, 20:509--536, 1981.

\bibitem{Lipton:1998}
R.~Lipton.
\newblock Heat conduction in fine scale mixtures with interfacial contact
  resistance.
\newblock {\em SIAM Journal of Applied Mathematics}, 58:55--72, 1998.

\bibitem{Malte:Bohm:2008}
A.~Malte and M.~B{\"{o}}hm.
\newblock Different choices of scaling in homogenization of diffusion and
  interfacial exchange in a porous medium.
\newblock {\em Mathematical Methods in the Applied Scinces}, 31:1257--1282,
  2008.

\bibitem{Timofte:2013}
C.~Timofte.
\newblock Multiscale analysis of diffusion processes in composite media.
\newblock {\em Computers and Mathematics with Applications}, 66:1573--1580,
  2013.

\end{thebibliography}
